\documentclass[11pt]{article}

\usepackage{fullpage}
\usepackage{amsthm}
\usepackage{amsmath, amssymb, graphicx, url}
\usepackage{color}
\usepackage{todonotes}
\usepackage{dsfont}

\usepackage{algorithm, algpseudocode} 
\usepackage{authblk}
\usepackage{hyperref}
\usepackage{mathtools}
\usepackage{changes}
\usepackage{ifthen}
\usepackage[normalem]{ulem}

\mathtoolsset{showonlyrefs=true}

\definecolor{refcol}{rgb}{0.1,0,0.7}
\hypersetup{colorlinks}
\hypersetup{linkcolor=refcol, citecolor=refcol}

\newcommand{\norm}[2][]{\| #2 \|_{#1}}
\newcommand{\normc}[2][]{\left\| #2 \right\|_{#1}}

\newtheorem{theorem}{Theorem}[section]

\newtheorem{corollary}[theorem]{Corollary}

\newtheorem{definition}[theorem]{Definition}

\newtheorem{lemma}[theorem]{Lemma}

\newtheorem{proposition}[theorem]{Proposition}

\newtheorem{assumption}[theorem]{Assumption}

\newcommand{\E}[1]{\mathbb{E}[#1]}
\newcommand{\Ec}[1]{\mathbb{E}\left[#1\right]}

\newcommand{\Pc}[1]{\mathbb{P}\left(#1\right)}
\newcommand{\bbE}{\mathbb{E}}

\newcommand{\Wj}{W^{j}}
\newcommand{\tjj}{\theta^{j,J}}
\newcommand{\C}[1]{\mathcal{C}_{\beta}(#1)}
\newcommand{\M}[1]{\mathcal{M}_{\beta}(#1)}

\newcommand{\Rh}{\rho_t^J}

\newcommand{\tm}{\overline{\theta}}

\newcommand{\ofb}{\omega^f_\beta}

\newcommand{\Li}{\lambda^{-1}}

\renewcommand{\i}[2]{\left\langle #1, #2 \right\rangle}

\newcommand{\supi}[1]{\underset{#1}{\sup}}
\newcommand{\sumj}{\frac{1}{J}\underset{j=1}{\overset{J}{\sum}}}
\newcommand{\sumi}[2]{\underset{#1}{\overset{#2}{\sum}}}

\newcommand{\mt}[1]{|#1|^2}
\newcommand{\mtp}[1]{|#1|^{2p}}

\newcommand{\R}{\mathbb{R}}
\newcommand{\PM}{\mathcal{P}}
\newcommand{\via}[2]{\overset{\text{#1}}{#2}}
\newcommand{\Tr}[1]{\text{Tr}(#1)}

\newcommand{\F}[1]{\mathbb{F}_{\varphi,t}(#1)}

\renewcommand{\L}[1]{L^{#1}}
\newcommand{\Lconv}[1]{\overset{\L{#1}}{\to}}

\DeclareMathOperator{\diag}{diag}
\DeclareMathOperator{\LV}{\mathcal LV}

\definecolor{jz}{rgb}{0,0,0}

\newcommand{\jz}[1]{{\color{jz}{#1}}}

\definecolor{mk}{rgb}{0,0,0}

\newcommand{\mk}[1]{{\color{mk}{#1}}}

\definecolor{sw}{rgb}{0,0,0}

\newcommand{\sw}[1]{{\color{sw}{#1}}}

\title{On the mean-field limit of consensus based methods}
\author[1]{Marvin Koß}
\author[2]{Simon Weissmann}
\author[1]{Jakob Zech}
\date{\today}

\affil[1]{\normalsize
  Universit\"at Heidelberg, Interdisziplin\"ares Zentrum f\"ur Wissenschaftliches Rechnen, D-69120 Heidelberg, Germany\\
 \texttt{koss@cl.uni-heidelberg.de}, \texttt{jakob.zech@uni-heidelberg.de}
}
\affil[2]{\normalsize
  Universit\"at Mannheim, Mathematisches Institut, D-68159 Mannheim, Germany\\
\texttt{simon.weissmann@uni-mannheim.de}
}

\begin{document}
\maketitle

\begin{abstract}
Consensus based optimization (CBO) employs a swarm of particles evolving as a system of stochastic differential equations (SDEs). Recently, it has been adapted to yield a derivative free sampling method referred to as consensus based sampling (CBS). In this paper, we investigate the ``mean field limit'' of a class of consensus methods, including CBO and CBS. This limit allows to characterize the system's behavior as the number of particles approaches infinity. Building upon prior work such as \cite{huangcbo}, we establish the existence of a unique, strong solution for these finite-particle SDEs. We further provide uniform moment estimates, which allow to show a Fokker-Planck equation in the mean-field limit. Finally, we prove that the limiting McKean-Vlasov type SDE related to the Fokker-Planck equation admits a unique solution.
\end{abstract}

\section{Introduction}
Consensus based optimization (CBO) methods comprise a class of stochastic interacting particle systems utilized as metaheuristic optimization techniques for non-convex and high-dimensional optimization problems. Metaheuristic methods, which aim to design efficient algorithms for solving challenging optimization problems effectively, have garnered significant attention in recent years \cite{pso, cbo,simulatedannealing,bonyadi}. These methods can be tailored to a wide array of optimization challenges, rendering them suitable for diverse applications, including machine learning, data science and engineering, among others.

Interacting particle optimization methods, e.g.\ \cite{pso,poli},
blend the exploration of complex landscapes inherent to the underlying cost function with the exploitation of experiential knowledge stored within the particle system. Through interactions among the particles, these methods have demonstrated promising success in avoiding local minima and discovering global solutions.

At 
its core, CBO, as introduced in \cite{cbo}, can be motivated as a system of multiple agents that collaboratively work toward finding the optimal solution, with the ultimate goal of reaching a consensus. More precisely, it considers the minimization problem
\begin{equation}
\min_{x\in\R^d}\ f(x),
\end{equation}
where $f:\R^d\to\R$ is a given cost function. The stochastic dynamical system employed by CBO relies solely on the evaluation of this cost function and remains derivative-free. In other words, the scheme does not involve the computation of (potentially costly or unobtainable) gradients.
In comparison to other particle swarm optimization methods, CBO also avoids the need for computing $\arg\min_{j=1,\dots,J}\ f(x^{j})$ across the particle system $(x^{j})_{j=1,\dots,J}$. Instead, it introduces a smooth approximation through a weighted average over the particle system. The weight function depends on $f$ and a tuning parameter $\beta>0$, and is defined by
\begin{equation}
    \label{eq:laplacefunctional}
    \ofb:=\begin{cases}
      \mathbb R^d \to \mathbb R\\
      {x}\mapsto \exp{(-\beta f(x))}.
    \end{cases}
\end{equation}
By construction, the weight function gives more weight to $x\in\R^d$ with low value of the cost
function $f(x)$. We define
$\mathcal M_\beta: \mathcal P (\R^d) \to \R^d$ with
\begin{equation}\label{eq:weightedaverage}
    \M{\eta}:= \frac{\int_{\R^d}x e^{-\beta f(x)}\mathrm{d}\eta(x)}{\int_{\R^d}e^{-\beta f(x)}\mathrm{d}\eta(x)}
\end{equation}
as the operator mapping a probability measure $\eta\in\mathcal P(\R^d)$ to its weighted average. The weighted average of the interacting particle system $(x^{j})_{j=1,\dots,J}$ can then be computed as
\[ \M{\eta^J}= \frac{1}{\sum_{j=1}^J e^{-\beta f(x^{j})}}\sum_{j=1}^J x^{j} e^{-\beta f(x^{j})}\, , \]
where $\eta^J=\frac1J\sum_{j=1}^J \delta_{x^{j}}\in\mathcal{P}(\R^d)$ denotes the empirical measure over the particle system. In the limit for $\beta\to\infty$ one may exploit the fact that
\cite{largedeviations}
\begin{align}
    \underset{\beta\to\infty}{\lim}{-\frac1\beta \log \int_{\R^d}\ofb(x)\mathrm{d}\eta(x)} &= \underset{x\in\text{supp}(\eta)}{\inf} f(x)\,
\end{align}
to justify the smooth approximation
\[\M{\eta^J} \approx \underset{x^{j,J},j=1,\dots,J}{\arg\min}\ f(x^{j,J}). \]
CBO \mk{uses this idea} to evolve a swarm of $J$ particles $(\theta^{j,J}_t)_{j=1,\ldots,J}$ according to competing terms for both local exploration around the current optimum of the particles and global exploration away from it \mk{by obtaining the optimum through the application $\M{\rho^J_t}$, where $\rho^J_t:=\frac1J\sum_{j=1}^J \delta_{\theta^{j,J}_t}$}. 
Local exploration is thus governed by moving the particles to the weighted average defined in \eqref{eq:weightedaverage}, while global exploration is incorporated through random pertubation depending on the distance of the particles from it. As introduced in \cite{cbo}, the interacting particle system of CBO is described as a coupled system of stochastic differential equations (SDEs) of the form
\begin{equation}\label{eq:original_CBO}
d\theta_t^{j,J} = -(\theta_t^{j,J}-\M{\Rh})\, \mathrm{d}t + \lambda^{-1} \diag_d(|\theta_t^{j,J}-\M{\Rh}|)\, dW_t^{j}\, 
\end{equation}
with $j=1,\dots,J$ and ensemble size $J\ge2$, where $(W_t^{j})_{t\in[0,T]}$ are independent Brownian motions in $\R^d$. Here, 
for $x=(x_1,\dots,x_d)^\top\in\R^d$ we write
\begin{equation*}
\diag_d(x) = \diag(x_1,\dots,x_d)\in\R^{d\times d}\,,
\end{equation*}
to denote a diagonal matrix. We assume that the initial ensemble is drawn as iid samples of some initial distribution $\rho_0\in \mathcal P(\R^d)$. 
The heuristic behind the considered system of SDEs is that the particles will meet a consensus in $\M{\rho^J_t}$ which can potentially lead to attaining the global optimum of $f$.

The existing theoretical convergence analysis of the scheme is mainly characterized by the system's behavior in the large ensemble size limit ($J\to\infty$), known as the \textit{mean field limit}. In the mean field limit, CBO can be described by a McKean-Vlasov type SDE 
of the form
\begin{equation}\label{eq:McKean_original}
d\bar \theta_t = -(\bar \theta_t-\M{\rho_t})\, dt + \lambda^{-1} \diag_d(|\bar \theta_t-\M{\rho_t}|)\, dW_t\, 
\end{equation}
where $\lambda>0$ is fixed, 
$(W_t)_{t\in[0,T]}$ is a Brownian motion in $\R^d$, and for each $t\ge0$, $\rho_t$ denotes the law of the solution $\bar \theta_t$ itself. Through It\^o's formula, the evolution of the law $(\rho_t)_{t\in[0,T]}$ can be characterized as a weak solution of the Fokker-Planck equation
\begin{equation}\label{eq:FP_original}
    \frac{\partial \rho}{\partial t} = \nabla \cdot ((\theta-\M{\rho})\rho + \frac{1}{2\lambda}\diag_d(|\theta - \M{\rho}|^2)\nabla \rho ),
\end{equation}
with initial condition $\rho_0$. Both of these equations have been analyzed in terms of well-posedness in \cite{cboanalysis}. While the convergence behavior of CBO is characterized through the mean field limit \eqref{eq:McKean_original} and \eqref{eq:FP_original}, the practical implementation is based on the particle approximation \eqref{eq:original_CBO}. 
A convergence analysis of \eqref{eq:original_CBO} towards the mean field limit \eqref{eq:McKean_original} for $J\to\infty$ has been provided in \cite{huangcbo,gerber2023}. 

\subsection{General class of consensus based methods}

Recently, in \cite{cbs}, the authors introduced a more general class of consensus based methods that allow the scheme to be applied as both an optimization 
and a sampling method. 
For optimization, the particle system may reach a consensus at the global minimum of the considered cost function $f$, i.e.\ all particles coalesce in the minimum. 
For sampling, 
the consensus is characterized as Gaussian approximation of the stationary distribution $\rho_\ast\propto\exp(-f)$ of the considered mean field equation. As such the method lends itself for use in Bayesian inference, to generate approximate samples from a posterior distribution.

Beside the weighted mean $\mathcal M_\beta$, the proposed formulation relies on the weighted covariance matrix $\mathcal C_\beta: \mathcal P(\R^d)\to\R^{d\times d}$ defined as
\begin{equation*}
    \mathcal C_\beta(\eta) :=\frac{\int_{\R^d}(x-\M{\eta})\otimes (x-\M{\eta}) e^{-\beta f(x)}\mathrm{d}\eta(x)}{\int_{\R^d}e^{-\beta f(x)}\mathrm{d}\eta(x)}\,.
\end{equation*}
In the finite particle regime, the proposed algorithm can be written in the form
\begin{equation}\label{eq:CBS_SDE}
    \mathrm{d}\theta_t^{j,J} = -(\theta_t^{j,J}-\M{\Rh})\mathrm{d}t + \sqrt{2\Li\C{\Rh}}\mathrm{d}\Wj_t,
  \end{equation}
for $j=1,\dots,J$, where $(\Wj_t)_{t\in[0,T]}$ again denote independent Brownian motions. The corresponding mean field limit as $J\to\infty$ can again be written as SDE of McKean-Vlasov type
\begin{equation}\label{eq:CBS_McKean}
    \mathrm{d}\tm_t = -(\tm_t-\M{\rho_t})\mathrm{d}t + \sqrt{2\Li\C{\rho_t}}\mathrm{d}W_t,
\end{equation}
where for each $t\ge0$, $\rho_t = \mathbb P ^{\tm_t}$ denotes the law
of $\tm_t$.  One fundamental advantage of the considered formulation
is that the dynamical system remains invariant under affine
transformations. This property is particularly beneficial for sampling
methods based on Markov chain Monte Carlo methods, see, e.g.,
\cite{CF2010,GW2010}.

Depending on the choice of $\lambda$, the authors identify the scheme as an optimization method ($\lambda=1$) or a sampling method ($\lambda = (1+\beta)^{-1}$). This observation has been supported by a detailed theoretical analysis of the mean field limiting system. 
However, the proposed reformulation of consensus based methods in the form of \eqref{eq:CBS_SDE} leaves open questions regarding the well-posedness of the finite particle system as well as the well-defined limit $J\to\infty$. The goal of the present paper is to address this issue. 

\subsection{Related Work}
Various particle based optimizers have been proposed, including Particle Swarm Optimization \cite{pso} and Consensus Based Optimization \cite{cbo}, for which analyses of the mean field regime have been done, respectively, by \cite{mfpso, huangnote} and \cite{cboanalysis,huangcbo}.
For CBO in particular, an analysis in the finite particle regime has been done in \cite{HJK2020}. The authors in \cite{refId0} modify CBO for high-dimensional problems, \cite{STW2023} adapt it and other bayesian inference algorithms for gradient inference, and \cite{Borghi2023} apply CBO to multi-objective optimization. An overview over the research on CBO can be found for example in \cite{Totzeck2022}.

A recently proposed variant is Consensus Based Sampling (CBS) \cite{cbs}, which is amenable both to sampling and optimization and has successfully been applied to rare event estimation in \cite{althaus2023consensusbased}. 
Moreover, a kernelized variant has been proposed by \cite{kernelized}, which is able to deal with multimodal distributions and objective functions with multiple global minima, respectively. \sw{When it comes to the implementation of CBS, the bias due to discretization and Gaussian approximation can be eliminated using a Metropolization step as proposed by \cite{sprungk2023metropolisadjusted}.}

Upon completion of this manuscript we were made aware of the recent study \cite{gerber2023}, which also analyzes convergence to the mean-field limit for both CBO and CBS. Specifically, in Thm.\ 2.6 \& 2.12 of \cite{gerber2023} using a
synchronous coupling approach they show pathwise Wasserstein $p$-chaos \cite[Definition 4.1]{poc1}, which actually implies our pointwise empirical propagation of chaos \cite[Lemma 4.2]{poc1}. 
  Our independent analysis, conceived concurrently with theirs,
differs in that we employ a compactness argument for the proof, considering weak convergence of the empirical measures rather than 
a coupling approach.

\subsection{Our contribution}
We make the following contributions: In the finite particle regime, we make use of a stochastic Lyapunov function in order to 
show well-posedness of the system of SDEs \eqref{eq:CBS_SDE}; i.e.\ 
  we prove existence of unique and strong solutions. For the mean field limit, we verify tightness of the empirical measures along the lines of \cite{mfpso}, and provide a compactness argument which closely follows the arguments in \cite{huangcbo}. These arguments imply weak convergence of the empirical measures  and, using Skorokhod's Representation, $\Tilde{\mathbb{P}}$-a.s. convergence on some appropriate probability space.
Lastly, we adopt the use of the Leray-Schauder fixed point theorem from \cite{cboanalysis} to verify well-posedness of the McKean-Vlasov type SDE \eqref{eq:CBS_McKean} in the form of existence of a unique solution.
\paragraph{Structure} The paper is structured as follows. In Section~2 we investigate the well-posedness of the finite particle SDEs. In Section~3 we investigate the mean field limit of the system of SDEs, by first verifying a weak solution of the Fokker Planck Equation exists, and then continuing to show a unique solution to the associated McKean-Vlasov limiting SDE exists. We only sketch proofs in the main part and refer to the supplementary material for some more results and proofs of the Lemmas and Theorems given in the main part.

\paragraph{Notation}
We write $\mathbb P^X$ for the law of a random variable $X$.
For a vector $x\in\R^d$, we write $|x|$ for its Euclidean norm,
while
$\|C\|_F=\sqrt{\Tr{CC^T}}=\sqrt{\sum_{i=1}^d\sum_{j=1}^d (C_{ij})^2}$ denotes
the 
Frobenius norm of a matrix $C\in\R^{d\times d}$ and $|[C]|_p=\Tr{\sqrt{C^*C}^p}^{\frac1p}$ its Schatten $p$-Norm. Note that the Frobenius norm is compatible with the Euclidean norm. 
We call the space of positive semi-definite $d\times d$-Matrices $\mathbb S^d_\geq$.
Throughout the article, we write particles as $\theta_t^{j,J}$ to make explicit the dependence on the number of particles.
We furthermore write $\PM(\R^d)$ for the space of probability measures on $\R^d$, and for $p\in\mathbb N$, $\PM_p(\R^d)$ 
is the space of probability measures on $\R^d$ with finite $p$-th moment.
We write $C(\R^d)$ for the space of continuous functions from $\R^d\to\R$,
$C^p(\R^d)$ for the space of
$p\in\mathbb{N}$ times continuously differentiable functions from $\R^d\to\R$,
$C_c^p(\R^d)$ for
compactly supported functions in $C^p(\R^d)$, and
$C_b(\R^d)$ for bounded functions in $C(\R^d)$.
Additionally, we'll require the space $C([0,T];\R^d)$
of continuous functions from $[0,T]\to\R^d$.
For $p\ge 1$ we use $\L{p}(\R^d,\rho)$ (or $\L{p}(\rho)$)
to denote the usual Lebesgue spaces of functions from $\R^d\to\R$.

\section{Well-posedness of the Particle Approximation}
Before considering the mean-field limit analysis, we 
firstly ensure that the finite particle SDE \eqref{eq:CBS_SDE} admits a unique and strong solution. This is
the primary concern of this section.
Following
\cite{huangcbo,mfpso},
throughout this paper 
we make 
the following assumptions on the
cost function:

\begin{assumption} \label{asp:costfuns} 
  The cost function $f:\R^d\to\R$ satisfies:
\begin{enumerate}
        \item There exists a constant $\text{Lip}(f)>0$ such that for all $x,y \in\mathbb{R}^d$ it holds
        \begin{equation*}
            |f(x)-f(y)| \leq \text{Lip}(f)\cdot(|x|+|y|)\cdot|x-y|. 
        \end{equation*}
        \item The cost function $f$ is bounded from below with $-\infty<f_*:=\underset{x\in\R^d}\inf f(x)$ and there exists a constant $c_u>0$ such that for all $x\in\mathbb{R}^d$ it holds
        \begin{align}
            f(x)-f_* \leq c_u(1+|x|^2)\,.
        \end{align}
        \item There exist constants $c_l, M >0$ such that for all $x \in\mathbb{R}^d$ with $|x|>M$ it holds
        \begin{align}
            f(x)-f_* \geq c_l|x|^2\,.
        \end{align}
\end{enumerate}
\end{assumption}
We make a further assumption on the finiteness of the sixth moment of the initial measure.
\begin{assumption} \label{asp:moments}
    The initial measure $\rho_0$ satisfies 
    \sw{$\int_{\mathbb R^d}|x|^{6}\mathrm{d}\rho_0(x) < \infty$}, 
    that is $\rho_0\in\mathcal P_6(\mathbb R^d)$.
\end{assumption}

For a fixed weight parameter $\beta>0$, in the following we
  use the operator (cp.~\eqref{eq:laplacefunctional})
\begin{equation*}
    L_\beta :=\begin{cases} \PM(\R^d) \to \PM(\R^d) \\
      \eta \mapsto \ofb\eta/\|\ofb\|_{\L{1}(\eta)}
      \end{cases}
\end{equation*}
which maps a probability measure $\eta$ to the reweighted probability
measure denoted by $L_\beta\eta$. Note that
$\mathcal{M}_\beta(\eta)=\mathbb{E}_{x\sim
  L_\beta\eta}[x]$.  We next state two lemmas that are used
throughout all parts of the mean field limit proof, and in the proof
of well-posedness, too.  The first Lemma has been proven in
\cite[Lemma 3.3]{cboanalysis} and states that second moments of
reweighted probability measures $L_\beta \mu$ can always be bounded by
the second moments of $\mu\in \mathcal P_2(\R^d)$.

\begin{lemma}\label{lem:estimate_weightedmean1} Let $f$ satisfy Assumption \ref{asp:costfuns}  and
  let $\mu\in\PM_2(\R^d)$.
Then 
    \begin{equation*}
      |\mathcal M_\beta (\mu)|^2\le  \int_{\R^d} |x|^2 \mathrm{d}L_{\beta}\mu(x) \leq  b_1 + b_2\int_{\R^d}|x|^2\mathrm{d}\mu(x)\, ,
    \end{equation*}
    where
      $b_1 := M^2+b_2,$
      $b_2 := 2\frac{c_u}{c_l} (1+\frac{1}{\beta c_l}\frac{1}{M^2})$
    are positive reals
    depending on $f$ and $\beta$, but not on $\mu$.
\end{lemma}
The following result extends the estimate provided by Lemma~\ref{lem:estimate_weightedmean1} to the weighted covariance matrix. This property is crucial in our considered analysis, since the diffusion is given by $\mathcal C_\beta(\rho_t^J)$ and $\mathcal C_\beta(\rho_t)$ respectively. 
\begin{lemma} \label{lem:estimate_weightedcov} Let $f$ satisfy Assumption~\ref{asp:costfuns} 
\sw{and $\mu\in\PM_{2}(\R^d)$.}
Then 
\begin{align}
    \max\Big(\sw{\|\C{\mu}\|_F},\|\sqrt{\C{\mu}}\|^{\sw{2}}_F \Big) \leq \sw{b_1 + b_2 \int_{\mathbb R^d} |x|^2 \mathrm{d}\mu(x)} \,,
\end{align}
where $b_1$, $b_2>0$ are the constants from Lemma~\ref{lem:estimate_weightedmean1}.
\end{lemma}

We now provide global existence of a unique strong solution for the coupled system of SDEs \eqref{eq:CBS_SDE} via stochastic Lyapunov theory. We refer the reader to \cite[Section~3.4]{khasminskii}.  Let us rewrite the system of $d\cdot J$ coupled SDEs \eqref{eq:CBS_SDE} as a single SDE in the joint state space $\times_{j=1}^J \R^{d}$ of the form 
\begin{equation}\label{eq:single_SDE}
{\mathrm d}\theta_t = F(\theta_t)\,{\mathrm d}t + G(\theta_t)\,{\mathrm d}W_t,
\end{equation}
where $\theta_t = (\theta_t^{j})_{j=1,\dots,J}\in\R^{J\cdot d \times 1}$, $W_t = (W_t^{j})_{j=1,\dots,J}\in\R^{J\cdot d\times 1}$ and \sw{for $x=(x^1,\dots,x^J)^\top\in \R^{J\cdot d}$ we defined}
\begin{subequations}\label{eq:FG}
\begin{align}
F(x) &= (\M{\rho^J}-x^{j})_{j=1,\dots,J}\in\R^{J\cdot d\times 1},\quad G(x) = \diag(\sqrt{2\lambda^{-1}\C{\eta^J}})_{j=1,\dots,J}\in\R^{J\cdot d\times J\cdot d},
\end{align}
\end{subequations}
with \sw{empirical measure} $\eta^J := \frac1J\sum_{j=1}^J \delta_{x^{j}}$ and $\diag(B_j)_{j=1,\dots,J}$ denoting a block diagonal matrix with blocks $B_j$, $j=1,\dots,J$.
Moreover, let $(\mathcal F_t^J)_{t\ge0} = (\sigma(\theta_0,W_s,s\le t))_{t\ge0}$ be the canonical filtration generated by the Brownian motion $(W_t)_{t\ge0}$ and the initial state $\theta_0$.

In order to ensure unique existence of strong solutions for the SDE \eqref{eq:single_SDE}, following \cite[Thm.\ 3.5]{khasminskii}, we need to construct a function $V\in C^2(\R^{J\cdot d})$ satisfying
\begin{subequations}\label{eq:conditions}
\begin{equation}\label{eq:generator_condition}
  \LV(x) := \nabla V(x)\cdot F(x)+\frac12 \Tr{G^\top(x)\nabla^2V(x)G(x)}\le cV(x),
\end{equation}
for some $c>0$ and 
\begin{equation}\label{eq:limiting_condition}
\inf_{\|x\|>R}\ V(x) \to \infty\quad  \text{for}\ R\to\infty.
\end{equation}
\end{subequations}
Here, $\mathcal L$ is the generator of the Markov process
$(\theta_t)_{t\in[0,T]}$.  This leads us to our first main result, in
particular the one of well-posedness of the finite particle SDEs:
\begin{theorem} \label{thm:stability} 
Let $f$ satisfy Assumption~\ref{asp:costfuns} and let $\theta_0=(\theta_0^{j})_{j=1,\dots,J}$ be $\mathcal F_0^J$-measurable maps $\theta_0^{j}:\Omega\to\R^d$. For all $T\ge0$ there exists a unique strong solution $(\theta_t)_{t\in[0,T]}$ (up to $\mathbb P$-indistinguishability) of
the system of coupled SDEs \eqref{eq:CBS_SDE}.
\end{theorem}
For the proof of Theorem~\ref{thm:stability} we construct the Lyapunov
function \[V(x):=\frac{1}{2J}\sumi{j=1}J|x^{j,J}|^2\]
  and verify the two
conditions 
in \eqref{eq:conditions} using the previous two lemmas. Note that
the limiting condition \eqref{eq:limiting_condition} is obviously
satisfied.  Beside the existence of unique strong solutions, we also
derive estimates that are uniform in $J$ for the second,
fourth and sixth moments of the particle system in
Section~\ref{sec:supp:momentestimates}. These estimates are crucial to
derive the mean field limit.

\section{Mean-Field Limit}
\label{sec:mfl}
Taking the limit of the number of particles $J\to\infty$ in \eqref{eq:CBS_SDE}, we pass to the mean field limit of the dynamical system. This limit is given by the following SDE of McKean-Vlasov type: 
\begin{equation}\label{eq:CBS_McKean}
    \mathrm{d}\tm_t = -(\tm_t-\M{\rho_t})\mathrm{d}t + \sqrt{2\Li\C{\rho_t}}\mathrm{d}W_t,
\end{equation}
where for each $t\ge0$, $\rho_t = \mathbb P ^{\tm_t}$ denotes the law of $\tm_t$. Deriving the mean field limit in this setting is challenging, as we have nonlinear and nonlocal terms $\M{\rho_t}$ and $\sqrt{\C{\rho_t}}$ arising in the drift and diffusion, which depend on the law of the solution itself. By applying It\^{o}'s formula, one however expects that the law of the solution of \eqref{eq:CBS_McKean} satisfies the Fokker-Planck Equation 
\begin{align}
\label{eq:fpe}
    \frac{\partial \rho}{\partial t} &= \nabla \cdot ((\theta-\M{\rho})\rho + \Li \C{\rho}\nabla \rho ).
\end{align}
We follow the arguments in
\cite[Section 3]{huangcbo}, where the mean field limit is obtained for the system of SDEs \eqref{eq:original_CBO}. For each $\varphi\in C_c^2(\R^d)$ and $t\in[0,T]$,
we define the functional $\mathbb{F}_{\varphi,t}$ on $\mathcal{C}([0,T];\PM(\R^d))$ as
\begin{align}\label{eq:Fvarphi}
    \F{\rho} :&= \int_{\R^d}\varphi(x) \mathrm{d}\rho_t(x) - \int_{\R^d} \varphi(x) \mathrm{d}\rho_0(x) + \int_0^t \int_{\R^d} \mathcal{L} \varphi(x) \mathrm{d}\rho_s(x) \mathrm{d}s \\
    &= \int_{\R^d}\varphi(x) \mathrm{d}\rho_t(x) - \int_{\R^d}\varphi(x) \mathrm{d}\rho_0(x) + \int_0^t \int_{\R^d}\sw{(x - \M{\rho_s})^{\top}}\nabla\varphi(x) \mathrm{d}\rho_s(x) \mathrm{d}s \\
    &\;- \Li \int_0^t \sumi{i=1}{d}\sumi{k=1}{d} \int_{\R^d}\C{\rho_s}_{ik}\frac{\partial^2}{\partial x_i\partial x_k} \varphi(x) \mathrm{d}\rho_s(x) \mathrm{d}s\,.
\end{align}

We then provide the definition of a \textbf{weak solution} to the FPE \eqref{eq:fpe} as defined in \cite[Def. 3.1]{huangcbo} by $\rho\in C([0,T];\PM_2(\R^d))$ satisfying the two conditions
    \begin{enumerate}
        \item For all $\phi \in{C}_b(\R^d)$ and $t_n\to t$ it holds
        \begin{equation}
            \label{eq:time_continuity_condition}
            \int_{\R^d} \phi(x)\mathrm{d}\rho_{t_n}(x) \to \int_{\R^d} \phi(x)\mathrm{d}\rho_{t}(x) 
        \end{equation}
        \item For all $\varphi\in{C}_c^2(\R^d)$ and $t\in[0,T]$ holds 
        \begin{equation}
            \label{eq:functional_condition}
            \F{\rho} = 0,
        \end{equation}
    \end{enumerate}

\sw{By this definition of weak solution we restrict to elements in $C([0,T];\mathcal P(\R^d))$ limiting the type of convergence result to be expected. We make the following distinction between pointwise and pathwise convergence:}

\begin{definition} \label{def:pathvspoint} \cite[Def. 3.25]{poc1} 
\sw{Let $\rho^J:=\frac1J\sum_{j=1}^J\delta_{\theta^{j,J}} \in \mathcal{P}(C([0,T], \R^d))$ be the empirical measure for a particle stochastic process $\{(\theta^{j,J})_{t\in[0,T]}\}_{j=1}^J$ of ensemble size $J$ 
and let $e_t:C([0,T],\R^d)\to\R^d,f\mapsto f(t)$ be the evaluation map. We define the time marginals $\rho_t^J:=(e_t)_\#\rho^J=\rho^J(e_t^{-1}(\cdot))$. Then $\mathbb{P}^{\rho^J}\in \mathcal{P}(\mathcal{P}(C([0,T],\R^d)))$ and $(\mathbb{P}^{\rho^J_t})_{t\in[0,T]}\in C([0,T],\mathcal{P}(\mathcal{P}(\R^d)))$ and we define the following types of convergence:}
\begin{enumerate}
    \item \sw{We say $\{(\mathbb{P}^{\rho^J_t})_{t\in[0,T]}\}_{J\in\mathbb N} \subset C([0,T],\mathcal{P}(\mathcal{P}(\R^d)))$ converges pointwise,} if there exists $(\rho_t)_{t\in[0,T]} \in C([0,T],\mathcal{P}(\R^d))$ such that $\mathbb{P}^{\rho^J_t} \overset{J\to\infty}{\to} \sw{\delta}_{\rho_t}$ in $\mathcal{P}(\mathcal{P}(\R^d))$, with $\sw{\delta}_{\rho_t}$ being a Dirac measure, for every $t\in[0,T]$,
    \item \sw{We say $\{\mathbb{P}^{\rho^J}\}_{J\in\mathbb N}\subset \mathcal{P}(\mathcal{P}(C([0,T],\R^d)))$ converges pathwise}, if there exists $\rho\in \mathcal{P}(C([0,T],\R^d))$ such that $\mathbb{P}^{\rho^J} \overset{J\to\infty}{\to} \sw{\delta}_\rho$, and $\sw{\delta}_\rho$ is a Dirac measure in $\mathcal{P}(\mathcal{P}(C([0,T],\R^d)))$.
\end{enumerate} 
The convergence is weak convergence. \sw{We call elements of $C([0,T],\mathcal{P}(\R^d))$ measure flows.} 
\end{definition} 

Pathwise convergence is stronger than pointwise convergence in the sense that pathwise implies pointwise, but not vice versa. We begin the proof of our following main result, Theorem \ref{thm:main}, by working in $\mathcal{P}(\mathcal{P}(C([0,T];\R^d)))$ in steps 1 and 2, which take place in Section \ref{sec:tight}, and then project to $C([0,T];\mathcal{P}(\mathcal{P}(\R^d)))$ for steps 3 and 4 (transferring the previously obtained weak convergence to the time marginals using \cite[Lemma 2.3 (2)]{huangcbo}), identifying the limit of the sequences in $\mathcal{P}(\mathcal{P}(\R^d))$
for all $t\in[0,T]$. Therefore, our final result yields \textit{pointwise} convergence. 

\sw{In order to guarantee uniqueness of the FPE \eqref{eq:fpe}, we make the following assumption on the solution of the corresponding McKean SDE \eqref{eq:CBS_McKean}:
\begin{assumption} \label{assumpt:evs} \sw{Let $(\rho_t)_{t\in[0,T]}$ be the 
the solution to the McKean SDE \eqref{eq:CBS_McKean}. We assume} that there exists $\sw{\bar \sigma}>0$ such that 
\[\sw{\C{\rho_t} - \bar\sigma\, {\mathrm{Id}} \succ 0}\]
\sw{for all $t\in[0,T]$}.
\end{assumption}
We note that for Gaussian initial distributions $\rho_0 = \mathcal N(m_0,C_0)$, Assumption~\ref{assumpt:evs} is satisfied for both CBO ($\lambda=1$) and CBS ($\lambda = (1+\beta)^{-1}$) \cite{cbs}. Moreover, it is worth mentioning that in Theorem~\ref{thm:mckeanexistence}, we will verify the existence of a strong solution to the McKean SDE \eqref{eq:CBS_McKean}.}

We are now ready to formulate our main result, which is analogous to \cite[Theorem 3.3]{huangcbo}. In particular, we are able to extend the mean-field limit analysis for the system \eqref{eq:original_CBO} presented in \cite{huangcbo} to the stochastic particle system driven by \eqref{eq:CBS_SDE}. 

\begin{theorem} \label{thm:main} 
\sw{Let Assumptions~\ref{asp:costfuns}, \ref{asp:moments} and \ref{assumpt:evs} hold} and for $T>0$ \sw{and $J\ge2$} let $(\theta^{j,J}_t)_{t\in[0,T]}$ be the unique solution of \eqref{eq:CBS_SDE} \sw{with empirical measures $\rho_t^J=\frac1J \sumi{j=1}{J}\delta_{\theta_t^{j,J}}$, $t\in[0,T]$,} and $\rho_0^{\otimes J}$-distributed initial data $(\theta^{j,J}_0)_{j=1,\ldots,J}$. 
\sw{Then $\{(\mathbb{P}^{\rho^J_t})_{t\in[0,T]}\}_{J\in\mathbb N}$ converges pointwise to $(\delta_{\rho_t})_{t\in[0,T]}$, where $(\rho_t)_{t\in[0,T]}\in C([0,T];\mathcal P(\R^D))$ is the (deterministic) unique weak solution of the FPE \eqref{eq:fpe}.}
\end{theorem}
    The proof of Theorem~\ref{thm:main} follows a compactness argument for verifying existence of the weak solution to the FPE and 
    can be sketched as follows:
\begin{enumerate}
    \item We show in Lemma~\ref{lem:tight} that the sequence of measures $\mathfrak F = \{\mathbb P^{\rho^J}\}_{J\in\mathbb N}$ is tight in $\mathcal{P}(\mathcal{P}({C}([0,T];\R^{D})))$.
    \item Prohorov's Theorem implies that $\mathfrak F$ is weakly relatively sequentially compact.
    \item If a subsequence $\{\rho^{J_k}\}_{k\in\mathbb N}$ is convergent in distribution\footnote{Equivalent to weak convergence of the $\mathbb P^{\rho^{J_k}}$.}, \mk{the limit of $\{(\rho^{J_k}_t)_{t\in[0,T]}\}_{k\in\mathbb N}$} solves the FPE weakly, which we show in Theorem~\ref{thm:weakly_conv_fpe}. In this step we employ Skorokhod Representations of the $\mathbb P^{\rho^{J_k}_t}$ and work with $\Tilde{\mathbb{P}}_t$-a.s. convergence on some appropriate space.
    \item \mk{Corollary~\ref{cor:uniqueness} implies that the limits of all $\{(\mathbb P^{\rho^{J_k}_t})_{t\in[0,T]}\}_{k\in\mathbb N}$ for weakly convergent subsequences $\{\mathbb P^{\rho^{J_k}}\}_{k\in\mathbb N}$ of $\mathfrak F$ are equal.}
    \item \mk{Because every subsequence has a weakly convergent subsequence, and all convergent subsequences attain the same limit, so does the entire sequence.}
\end{enumerate}
This line of arguments is in principle analogous to
  \cite{huangcbo}.
  Nevertheless, our analysis significantly deviates from this work as we
  consider a complete covariance matrix rather than a diagonal diffusion term. This results in more intricate computations and estimations.
We proceed with the first step of the above procedure, which is proving tightness of the sequence of laws, for which the following moment estimates and propagation of chaos \cite[Chapter~8]{del2004feynman} will be instrumental.

\subsection{Moment Estimates}\label{sec:supp:momentestimates}
We now provide some moment estimates, which are fundamental to multiple steps of the compactness argument we present for Theorem \ref{thm:main}. These estimates include the second, fourth and sixth moments along the lines of \cite[Lemma 3.4]{cboanalysis}.

\begin{lemma} \label{lem:momest} Let $\theta_0 = (\theta_0^{j})_{j=1,\dots,J}$ be $\mathcal F_0^J$-measurable such that $\rho_0$ satisfies Assumption \ref{asp:moments}. Then, for $p=1,2,3$ there exists a constant $K_p(T)>0$, depending on $T$ and $p$ but independent of $J$, such that 
    \begin{equation}
        \supi{t\in[0,T]} \Ec{\int_{\R^d}\mtp{x}\mathrm{d}\Rh(x)} \leq K_p(T).
    \end{equation}
\end{lemma}
    
\subsection{Tightness of the empirical measures}\label{sec:tight}
The following theorem provides the tightness of the empirical measures using the criteria of Aldous' \cite[Thm 16.10]{billingsley}, which we recall in the supplementary material in Theorem~\ref{thm:aldous}. We adapt the strategy of the proof of \cite[Theorem~3.3]{mfpso}. Tightness ensures the particles do not tend to infinity with notable probability and is fundamental to the compactness argument.

\begin{lemma} \label{lem:tight} 
    Let $f$ satisfy Assumption \ref{asp:costfuns} and let $\rho_0$ satisfy Assumption \ref{asp:moments}. Furthermore, let $\{(\theta^{j,J})_{t\in[0,T]}\}_{j=1}^J$ be the unique solution to \eqref{eq:CBS_SDE} with $\rho_0^{\otimes J}$-distributed $\{\theta^{j,J}_0\}_{j=1}^J$. Then $\mathfrak F =\{\mathbb P^{\rho^J}\}_{J\in\mathbb N}$ is tight in $\mathcal{P}(\mathcal{P}(C([0,T];\R^{D})))$.
\end{lemma}

\subsection{Convergence of the empirical measures}

Due to Prokhorov’s theorem \cite[Theorem 6]{billingsley} it \sw{is a direct consequence of Lemma~\ref{lem:tight} that there exists a subsequence 
$\{\mathbb P^{\rho^{J_k}}\}_{k\in\mathbb N}$ which converges weakly (i.e.~$\{\rho^{J_k}\}_{k\in\mathbb N}$ converges in distribution) to some random measure $\rho:\Omega \to \mathcal{P}(C([0,T];\R^{D}))$. By Skorokhod’s lemma \cite[Thm 6.7]{billingsley} we then find a common probability space $(\tilde\Omega,\tilde{\mathcal F},\tilde{\mathbb P})$ on which $\{\rho^{J_k}\}_{k\in\mathbb N}$ converges $\tilde{\mathbb P}$-almost surely as random variable to a random variable $\rho:\tilde\Omega \to \mathcal{P}(C([0,T];\R^{D}))$ with values in $\mathcal{P}(C([0,T];\R^{D}))$.} 
It is important to note that at this point in our overall argument \sw{we 
project} from $\mathcal P(C([0,T],\R^d))$ to $C([0,T],\mathcal P(\R^d))$ by use of the map
\begin{align}
    \Pi : \mathcal P(C([0,T], \R^d)) &\to C([0,T],\mathcal P(\R^d)) \\
    \rho &\mapsto (\rho_t)_{t\in[0,T]} 
\end{align}
defined in \cite[Sec. 3.4.3]{poc1} and work with the measure flows 
\begin{equation}
    \overline\rho^J:=\Pi(\rho^J)=(\rho^J_t)_{t\in[0,T]}\,,
\end{equation}
\mk{where the time marginals $\rho_t^J$ are as in Definition~\ref{def:pathvspoint}.}
\sw{This 
enables us to verify} the 
two pointwise convergences \sw{described below}. The weak convergence of the time marginals is then obtained via \cite[Lemma 2.3 (2)]{huangcbo}.

In continuing with our compactness argument for the unique solvability of the FPE, we now verify condition \eqref{eq:functional_condition} for \sw{the limit of} \mk{a} convergent subsequence $\{\rho^{J_k}\}_{k\in\mathbb N}$ (which we simply denote by the full sequence $\{\rho^{J}\}_{J\in\mathbb N}$) 
using $\L{1}$- and $\L{2}$-convergence \sw{with respect to the probability measure $\tilde{\mathbb P}$ on the common probability space $(\tilde\Omega,\tilde{\mathcal F},\tilde{\mathbb P})$}, i.e. $\forall t\in[0,T]$ and arbitrary $\varphi\in C_c^2(\R^d)$:

\begin{enumerate}
    \item $\F{\overline \rho^J} \Lconv{2} 0$ as $J\to\infty$,
    \item $\F{\overline \rho^J} \Lconv{1} \F{\overline \rho}$ as $J\to\infty$.
\end{enumerate}
Since the limits are $\tilde{\mathbb P}$-almost surely unique, we obtain that the limiting random measure $\rho_t$ satisfies $\F{\rho}=0$, $\tilde{\mathbb P}$-almost surely. 
These two points are shown in the following two Lemmas.

The first point is shown using 
\jz{$\L{2}$}-convergence.
By the reasoning in \cite[Prop. 3.2]{huangcbo}, we have the following result:

\begin{lemma} \label{lem:funcconvzero}
    Let $f$ satisfy {Assumption}~\ref{asp:costfuns} and let $\rho_0$ satisfy Assumption \ref{asp:moments}. For $J\in\mathbb{N}$, assume that $\{(\tjj_t)_{t\in[0,T]}\}_{j=1,\ldots,J}$ is the unique strong solution to the particle system \eqref{eq:CBS_SDE} with $\rho_0^{\otimes J}$-distributed initial data $\{\tjj_0\}_{j=1,\ldots,J}$. For the projections $\{\overline\rho^J\}_{J\in\mathbb{N}}$ of the empirical measures it then holds for all $t\in[0,T]$ and $\varphi\in C_c^2(\R^d)$ that
    \begin{equation}
        \F{\overline\rho^J} \Lconv{2} 0,
    \end{equation}
    \mk{where $\overline\rho^J\in C([0,T],\mathcal P(\R^d))$.}
\end{lemma}
One obtains this result by applying It\^o's Lemma to the functional and using Lemma \ref{lem:momest}. The second part of the limiting procedure is shown in the following Lemma, which is analogous to \cite[Theorem 3.3]{huangcbo}:
\begin{lemma} \label{lem:funcconvlimit}  \label{lem:l1convFtoLimit}
    Let $f$ satisfy Assumption \ref{asp:costfuns} and let $\rho_0$ satisfy Assumption \ref{asp:moments}. For $J\in\mathbb{N}$, assume that $\{(\tjj_t)_{t\in[0,T]}\}_{j=1,\ldots,J}$ is the unique strong solution to the particle system \eqref{eq:CBS_SDE} with $\rho_0^{\otimes J}$-distributed initial data $\{\tjj_0\}_{j=1,\ldots,J}$. 
    \mk{Let $\{\rho^J\}_{J\in\mathbb{N}}\subset \mathcal P(C([0,T],\R^d))$ be a weakly convergent subsequence of the empirical measures, denoted by the full sequence, and let $\rho$ be their limit. It then holds for the projections} for all $t\in[0,T]$ and $\varphi\in C_c^2(\R^d)$ that
    \begin{align}
        \F{\overline\rho^J} \Lconv{1} \F{\overline\rho}.
    \end{align}
\end{lemma}
In order to obtain this result, one inspects the summands of the functional one by one for convergence.
The previous two lemmas permit us to obtain the following result:
\begin{theorem} \label{thm:weakly_conv_fpe}
    Let $f$ satisfy {Assumption}~\ref{asp:costfuns} and let $\rho_0$ satisfy Assumption \ref{asp:moments}. For $J\in\mathbb{N}$, assume that $\{(\tjj_t)_{t\in[0,T]}\}_{j=1,\ldots,J}$ is the unique strong solution to the particle system \eqref{eq:CBS_SDE} with $\rho_0^{\otimes J}$-distributed initial data $\{\tjj_0\}_{j=1,\ldots,J}$. Let $\{\rho^J\}_{J\in\mathbb{N}}$ be a weakly convergent subsequence of the  empirical measures, denoted by the full sequence, with  limit \mk{$\rho\in \mathcal P(C([0,T],\R^d))$}. \mk{Then the projection of the limit $\overline\rho\in C([0,T],\mathcal P(\R^d))$} is a weak solution of the FPE \eqref{eq:fpe}. 
\end{theorem}
The above result provides that the limiting measure of convergent subsequences indeed (weakly) solves the FPE.

\subsection{Properties of the McKean-Vlasov SDE}

For our proof of the unique solvability of the FPE \eqref{eq:fpe} we require existence and uniqueness of the solution to the following associated SDE of McKean-Vlasov type:

\begin{equation}
    \label{eq:mckeansde}
    \mathrm{d}\overline{\theta}_t = -(\overline{\theta}_t-\M{\rho_t})\mathrm{d}t + \sqrt{2\Li \C{\rho_t}}\mathrm{d}W_t\,,
\end{equation}
where $\rho_t = \mathbb{P}^{\overline{\theta}_t}$ denotes the law of $\overline{\theta}_t$.
This SDE describes the behavior of a representative particle in the mean field limit, and if the solution exists, its law solves the FPE \eqref{eq:fpe}. The existence and solution are nontrivial to obtain since the drift and diffusion term depend in nonlinear ways on the law of the solution itself. Hence, we dedicate the following Theorem to the existence and the one after it to the uniqueness of the solution to this SDE. We follow the strategy of \cite[Theorem 3.2]{cboanalysis}, wherein the authors obtain unique solvability for the McKean SDE for CBO using Schaefer's Fixed point Theorem (SFPT)\footnote{cf. Thm \ref{thm:schaefersfpt} in the Appendix}. The difference in our analysis lies in the fact that we have the covariance matrix of the reweighted measure instead of a diagonal diffusion matrix. We emphasize that this extension is non-trivial.

Next, we obtain estimates that will be useful in proving both existence and uniqueness of the solution to the McKean-Vlasov type SDE and is made in analogy to \cite[Lemma 3.2]{cboanalysis}. 
\begin{lemma} \label{lem:substitutestability}
    Let \sw{$\mu$, ${\nu}\in\mathcal{P}_6(\R^d)$}  with \[\Big(\int_{\R^d} |x|^6\, \mathrm{d}\mu(x)\Big)^{\frac16},~~\jz{\Big(}\int_{\R^d} |\widehat x|^6 \mathrm{d}\nu(\widehat x)\Big)^{\frac16}\le K\, .\] 
    Then there exist \sw{$c_0,c_1>0$} such that
        \begin{enumerate}
        \item[(i)] $\sw{|\M{\mu}-\M{\nu}| \leq c_0 W_2(\mu,{\nu})}$,
        \item[(ii)] $\|\C{\mu}-\C{\nu}\|_F  \leq c_1 W_2(\mu,{\nu})$,
        \end{enumerate}
    where \sw{$c_0$ and} $c_1$ depend on $\mu$, $\nu$ solely through $K$.
  \end{lemma}
  
We now provide the existence of the solution to the McKean-Vlasov type SDE using a fixed point argument which is made closely following the strategy of the proof in \cite[Theorem 3.1]{cboanalysis}.
\begin{theorem}\label{thm:mckeanexistence}
    Let $\overline\theta_0\sim\rho_0$ with $\rho_0$ satisfying Assumption \ref{asp:moments} and let $f$ satisfy Assumption~\ref{asp:costfuns}. For all $T\ge0$ there exists a strong solution $(\overline\theta_t)_{t\in[0,T]}$ of the McKean SDE \eqref{eq:mckeansde}.
\end{theorem}

\sw{Using Assumption~\ref{assumpt:evs}}, we verify that solutions to the McKean-Vlasov type SDE are unique in the following sense.
\begin{theorem}\label{thm:mckeanuniqueness}
     Let Assumption \ref{assumpt:evs} hold, let $\overline\theta_0\sim\rho_0$ with $\rho_0$ according to Assumption \ref{asp:moments} and let $f$ satisfy Assumption~\ref{asp:costfuns}. Then the solution of the McKean SDE \eqref{eq:mckeansde} is unique up to $\mathbb P$-indistinguishability. 
\end{theorem}

\subsection{Uniqueness of the FPE}
Until now we have been concerned with weakly convergent subsequences of the empirical measures, and shown that their limit is indeed a solution to the FPE. For the convergence of the whole sequence, it suffices to show that the solution of the FPE is unique. This is shown in Corollary \ref{cor:uniqueness}, which is analogous to \cite[Lemma 3.2]{huangcbo} and uses the results of the previous section, most crucially the previously obtained unique solvability of the associated McKean-Vlasov type SDE.

In order to prove the Corollary, we also need the following auxiliary result, which is analogous to \cite[Theorem 4.3]{huangcbo}:
\begin{proposition} \label{prop:linearizedfpeunique}
    For $T>0$, let $(u,C)\in C([0,T],\R^d\times \R^{d\times d})$ and $\rho_0$ according to Assumption \ref{asp:moments}. Then the following linear PDE 
    \begin{align}
    \label{eq:linearizedfpe}
        \frac{\partial \overline{\rho}}{\partial t} &= \Li \sumi{i=1}{d}\sumi{k=1}{d}\frac{\partial^2}{\partial x_i\partial x_k}[(C_t)_{ik}{\rho}_t] - \nabla\cdot[(x-u_t)\rho_t]
    \end{align}
    has a unique weak solution $\overline\rho\in C([0,T];\PM_2(\R^d))$. 
\end{proposition}

Using the above result for the unique solvability of a linearized version of the FPE, we are able to prove the pointwise uniqueness, which is stated in the following Corollary:
\begin{corollary} \label{cor:uniqueness} 
    Let $\rho_0$ satisfy Assumption \ref{asp:moments} and let it be the initial data of two weak solutions $\overline\rho^1, \overline\rho^2$ to the FPE \eqref{eq:fpe}. 
    \sw{Under Assumption~\ref{assumpt:evs}} we have
    \begin{align}
        \underset{t\in[0,T]}{\sup}W_2(\rho_t^1, \rho_t^2) &= 0.
    \end{align}
\end{corollary}

\subsection{Proof of Theorem~\ref{thm:main}}
By connecting the results of the previous subsections, we are ready to prove our main result, Theorem~\ref{thm:main}:

\begin{proof}[Proof of Theorem~\ref{thm:main}]
     Due to the assumptions on our initialization, we 
     \sw{are} able to apply propagation of chaos and the initially computed moment estimates, and 
     \sw{obtain with} Lemma~\ref{lem:tight}  that the sequence of measures $\mathfrak F = \{\mathbb P^{\rho^J}\}_{J\in\mathbb N}$ is tight in $\mathcal{P}(\mathcal{P}(C([0,T];\R^{D})))$. Subsequently applying Prohorov's Theorem implies weak relative sequential compactness of $\mathfrak F$. 
     Theorem \ref{thm:weakly_conv_fpe} \sw{shows} that if we take a subsequence of the $\{\rho^J\}_{J\in\mathbb N}$ to be convergent in distribution, the limit of the projections $\{\overline\rho^J\}_{J\in\mathbb N}$ solves the FPE weakly.
     Furthermore, Corollary \ref{cor:uniqueness} implies that the limit of \mk{the projections of } all weakly convergent subsequences of $\mathfrak F$ is equal, implying the weak convergence of the entire sequence $\{\overline\rho^J\}_{J\in\mathbb N}$. 
\end{proof}

\section{Conclusion}
After providing a stability result for the finite particle SDE, we investigated the mean field limit of the system \sw{of CBO and CBS}, successfully showing that it holds under moderate assumptions that are common in the literature. Specifically, we employed \sw{propagation of chaos via a} compactness argument by analyzing the weak convergence of the empirical measures, which additionally yielded that the limiting measure is not only a solution to the FPE, but in fact the only one.


\bibliographystyle{plain}
\bibliography{bibliography.bib}

\appendix
\onecolumn
\section{Proofs of Section~2}
\subsection{Proofs of preliminary results}
\begin{proof}[Proof of Lemma~\ref{lem:estimate_weightedcov}]
\sw{By linearity of the trace operator we have
\begin{align*}
        \|\sqrt{\C{\mu}}\|^{2}_F 
        &= \int_{\R^d}\Tr{(x-\M{\mu})\otimes(x-\M{\mu})}\mathrm{d}L_{\beta}\mu(x)\\
        &= \int_{\R^d}|x-\M{\mu}|^2\mathrm{d}L_{\beta}\mu(x)\\
        &= \int_{\R^d}|x|^2\mathrm{d}L_{\beta}\mu(x)-|\M{\mu}|^{2}\\
        &\leq \int_{\R^d}|x|^2\mathrm{d}L_{\beta}\mu(x)\,,
    \end{align*}
where we used $|x-y|^2 = |x|^2 - 2\langle x,y\rangle +|y|^2$, the definition of \jz{$\M{\mu}=\int_{\R^d}x \,\mathrm{d}L_\beta\mu(x)$} and the fact that $\M{\mu}$ is a constant independent of $x$. We apply Lemma~\ref{lem:estimate_weightedmean1} to deduce that 
\[\|\sqrt{\C{\mu}}\|^{2}_F \le \int|x|^2\mathrm{d}L_{\beta}\mu(x) \le b_1 + b_2 \int |x|^2 \mathrm{d}\mu(x)\,.  \]
Similarly, applying Cauchy-Schwarz inequality we obtain
\begin{align*}
    \| \C{\mu}\|_F^2 &= \Tr{\int_{\R^d\times \R^d} \langle x - \M{\mu},\hat x - \M{\mu} \rangle (x-\M{\mu})\otimes (\hat x-\M{\mu}) \mathrm{d} L_{\beta}\mu( x) \mathrm{d} L_\beta\mu(\hat x)}\\
    &= \int_{\R^d\times \R^d} \langle  x - \M{\mu},\hat x - \M{\mu} \rangle^2 \mathrm{d} L_{\beta}\mu( x) \mathrm{d} L_\beta\mu(\hat x)\\ 
    &\le \int_{\R^d\times \R^d} | x - \M{\mu}|^2 |\hat x - \M{\mu}|^2  \mathrm{d} L_{\beta}\mu( x) \mathrm{d} L_\beta\mu(\hat x)\\
    & = \left(\int_{\R^d} | x - \M{\mu}|^2 \mathrm{d} L_{\beta}\mu( x) \right)^2\\
    & \le \left(\int_{\R^d} | x|^2 \mathrm{d} L_{\beta}\mu( x) \right)^2\,,
\end{align*}
which yields again with Lemma~\ref{lem:estimate_weightedmean1} that
\[\| \C{\mu}\|_F \le b_1 + b_2 \int | x|^2 \mathrm{d}\mu( x)\,. \]
}
\end{proof}

\subsection{Proofs of main results}
\begin{proof}[Proof of Theorem~\ref{thm:stability}]
  The proof strategy is to apply \cite[Thm.\ 3.5]{khasminskii}.
    To this end we verify the assumptions of this theorem. First we
    note that the drift $F$ and diffusion $G$ defined in \eqref{eq:FG}
    are both locally Lipschitz and satisfy a linear growth condition,
    which follows straightforwardly from the Lipschitz continuity of
    $f$. For more details, we refer for example to \cite[Lemma 2.1]{cboanalysis} for a similar calculation which can be adapted up to minor modifications to the present setting.
    In the rest of the proof we construct a
    Lyapunov function $V$ and show that \eqref{eq:conditions} holds.
  
  Define
  \[V( x) = \frac{1}{2J}\sum_{j=1}^J 
    | x^{j}|^2\qquad\text{for
    } x=( x^{j})_{j=1,\dots,J}\in\R^{J\cdot d}.\]
  Condition \eqref{eq:limiting_condition} is clearly satisfied for
    this $V$,
    so that it only remains to show \eqref{eq:generator_condition}.
  We compute the generator and derive the following upper bound using
  Cauchy-Schwarz inequality
\begin{align*}
\LV( x) &= \frac{1}{J}\sum_{j=1}^J \langle  x^{j}, \M{\rho^J}- x^{j}\rangle + \frac{1}{2J}\sum_{j=1}^J 2\lambda^{-1}\Tr{\C{\rho^J}} \\
				&\le -\frac1J\sum_{j=1}^J | x^{j}|^2 + \frac1J\sum_{j=1}^J| \M{\rho^J}|| x^{j}|+ \frac1{J\lambda}\sum_{j=1}^J\Tr{\C{\rho^J}}.
\end{align*}
By Lemma~\ref{lem:estimate_weightedcov}, we have
\begin{equation*}
\frac1J\sum_{j=1}^J\Tr{\C{\rho^J}} \leq \sw{b_1+b_2\frac1J\sum_{i=1}^J| x^j|^2} \, 
\end{equation*}
and using Lemma~\ref{lem:estimate_weightedmean1}, we similarly obtain
\[|\mathcal M_{\beta}(\rho^J)|^2\le b_1 + b_2 \frac1J\sum_{i=1}^J | x^{i}|^2\,.\]
By Young's inequality we deduce that
\begin{align} \frac1J\sum_{j=1}^J|\mathcal M_\beta(\rho^J)|| x^{j}| &\le \frac1J\sum_{j=1}^J \left( \frac{|\mathcal M_\beta(\rho^J)|^2}2 +\frac{| x^{j}|^2}2\right)\le \frac{b_1}{2} + \frac{1+b_2}{2J}\sum_{j=1}^J| x^{j}|^2. 
\end{align}
Hence, there exist constants $c_1$, $c_2>0$
such that $\LV( x) \le c_1 + c_2 V( x)$
and thus $V_{c_1}( x) = c_1+V( x)$ gives the desired Lyapunov
function. In all this shows that the assumptions of \cite[Thm.\
3.5]{khasminskii} are satisfied. An application of this theorem
then implies the claim.
\end{proof}

\begin{lemma}\label{lem:auxiliary_moments}
  Let $p$, $q\ge 0$, let $\mu$ be a probability measure on $\R^d$, and
  let $f\in \L{p+q}(\R^d,\mu)$. Then
  \begin{equation*}
    \int_{\R^d}|f(x)|^p\,\mathrm{d}\mu(x)
    \int_{\R^d}|f(y)|^q\,\mathrm{d}\mu(y)
    \le
    \int_{\R^d}|f(x)|^{p+q}\,\mathrm{d}\mu(x).
  \end{equation*}
\end{lemma}
\begin{proof}
  Without loss of generality $p\le q$. For $p=0$ the statement is
  trivial, so let $p>0$. Then $q/p\ge 1$ and by H\"older's inequality
  \begin{equation*}
    \int_{\R^d}|f(x)|^p\,\mathrm{d}\mu(x)
    \le
    \left(\int_{\R^d}|f(x)|^{p\frac{q}{p}}\,\mathrm{d}\mu(x)\right)^{\frac{p}{q}}
    =\left(\int_{\R^d}|f(x)|^{q}\,\mathrm{d}\mu(x)\right)^{\frac{p}{q}}.
  \end{equation*}
  This estimate together with Jensen's inequality yields
  \begin{equation*}
\int_{\R^d}|f(x)|^p\,\mathrm{d}\mu(x)
    \int_{\R^d}|f(y)|^q\,\mathrm{d}\mu(y)    
    \le \left(\int_{\R^d}|f(x)|^q\,\mathrm{d}\mu(x) \right)^{\frac{q+p}{q}}
    \le \int_{\R^d}|f(x)|^{q+p}\,\mathrm{d}\mu(x),
  \end{equation*}
  which gives the claim.
\end{proof}

\begin{proof}[Proof of Lemma~\ref{lem:momest}]
  Set
    \begin{equation}
        g:=\begin{cases} \R \times \R^d \to \R\\
          (t, x) \mapsto 
          | x|^{2p}.
        \end{cases}
    \end{equation}
    For $ x=( x_k)_{k=1}^d\in\R^d$ we have
    \begin{align}
        \frac{\partial g}{\partial t}(t, { x}) &= 0, \\
      \frac{\partial g}{\partial \jz{ x_k}}(t, { x}) &= 2p x_k|x|^{2(p-1)},\\
      \frac{\partial^2 g}{\partial \jz{ x_k}^2}(t,  x) &= 2p| x|^{2(p-1)} + \mk{4}p(p-1)| x|^{2\,\sw{\max(0,p-2)}}\jz{ x_k^2}.
    \end{align}

    Now let $\theta_t^j\in\R^d$, $j=1,\dots,J$, be the solution of
      \eqref{eq:CBS_SDE}. Applying the It\^{o} formula to $g$ yields
    \begin{align}
        \mathrm{d}|\theta^j_t|^{2p} &= 2p|\theta^j_t|^{2(p-1)}\langle \theta_t^j,\mathrm{d}\theta^j_t \rangle  + p|\theta^j_t|^{2(p-1)}\langle \mathrm{d}\theta_t^j,\mathrm{d}\theta^j_t \rangle + 2p(p-1)|\theta^j_t|^{2(p-2)}\langle\theta_t^j,\mathrm{d}\theta^j_t\rangle \\
        &= 2p|\theta^j_t|^{2(p-1)}\i{\theta^j_t}{-(\theta^j_t-\M{\rho^J_t})\mathrm{d}t+\sqrt{2\Li\C{\rho^J_t}}\mathrm{d}W_t} \\
        &\quad + p|\theta^j_t|^{2(p-1)}\i{-(\theta^j_t-\M{\rho^J_t})\mathrm{d}t+\sqrt{2\Li\C{\rho^J_t}}\mathrm{d}W_t}{-(\theta^j_t-\M{\rho^J_t})\mathrm{d}t+\sqrt{2\Li\C{\rho^J_t}}\mathrm{d}W_t} \\
        &\quad + \mk{4}p(p-1)|\theta^j_t|^{2\max(0,p-2)} \i{\theta^j_t}{-(\theta^j_t-\M{\rho^J_t})\mathrm{d}t+\sqrt{2\Li\C{\rho^J_t}}\mathrm{d}W_t}\,.
    \end{align}
    Integrating over time and applying the expectation gives
    \begin{align}
        \E{|\theta^j_t|^{2p}} &= \bbE\Bigg[|\theta_0^j|^{2p} -\int_0^t2p|\theta^j_s|^{2(p-1)}\langle\theta^j_s,\theta^j_s-\M{\rho^J_s}\rangle \mathrm{d}s\\
        &\quad+ \int_0^t2p|\theta^j_s|^{2(p-1)}\langle \theta^j_s,\sqrt{2\Li\C{\rho_s^J}}\rangle \mathrm{d}W_s + \int_0^t 2 p|\theta^j_s|^{2(p-1)}\Li \Big\|\sqrt{\C{\rho^J_s}}\Big\|_F^2 \mathrm{d}s \\
        &\quad+ \mk{4}p(p-1)\Big(\int_0^t|\theta_s^j|^{2\max(0,p-2)}\langle \theta^j_s,\theta^j_s-\M{\rho^J_s}\rangle \mathrm{d}s+\int_0^t|\theta_s^j|^{2\max(0,p-2)}\langle\theta_s^j,\sqrt{2\Li\C{\rho^J_t}}\rangle \mathrm{d}W_s\Big)\Bigg].
    \end{align}
    Using the linearity of the expectation and the integral and the fact that the third and last integrands are martingales, we have
    \begin{align}
        \E{|\theta^j_t|^{2p}} &= \E{|\theta_0^j|^{2p}} - 2p\int_0^t\Ec{|\theta^j_s|^{2(p-1)}\langle\theta^j_s,\theta^j_s-\M{\rho^J_s}\rangle} \mathrm{d}s
        +2p\Li\int_0^t\Ec{|\theta^j_s|^{2(p-1)} \Big\|\sqrt{\C{\rho^J_s}}\Big\|_F^2}\mathrm{d}s \\
        &\quad+ 2p(p-1)\int_0^t\Ec{|\theta_s^j|^{2\max(0,p-2)}\langle \theta^j_s,\theta^j_s-\M{\rho^J_s}\rangle}\mathrm{d}s  \\
        &\leq  \E{|\theta_0^j|^{2p}}  +  \mk{4p^2}\int_0^t\Ec{(|\theta^j_s|^{2(p-1)}\mk{+|\theta^j_s|^{2\max(0,p-2)})}(|\theta^j_s|^2\jz{+|\theta^j_s-\M{\rho^J_s}|^2)}} \mathrm{d}s\\
        &\quad+2p\Li\int_0^t\Ec{|\theta^j_s|^{2(p-1)}\Big\|\sqrt{\C{\rho^J_s}}\Big\|_F^2}\mathrm{d}s\,.
    \end{align}
    Next, we sum over the particles and divide by $J$, after which we apply the inequality $|a+b|^2\le 2|a|^2+2|b|^2$ to obtain 
    \begin{align}
        \Ec{\int_{\R^d}| x|^{2p}\mathrm{d}\rho^J_t( x)}&\leq \Ec{\int_{\R^d}| x|^{2p}\mathrm{d}\rho_0^{\otimes J}( x)} \\
        &\quad + \mk{4}p^2\int_0^t\Ec{\int_{\R^d}(| x|^{2(p-1)}\mk{+ | x|^{2\max(0,p-2)})}(\mk{3| x|^2}+2|\M{\rho^J_s}|^2)\mathrm{d}\rho^J_s( x)} \mathrm{d}s
        \\ 
        &\quad +2p\Li\int_0^t\Ec{\int_{\R^d}| x|^{2(p-1)}\Big\|\sqrt{\C{\rho^J_s}}\Big\|^2_F\mathrm{d}\rho^J_s( x)}\mathrm{d}s \\
    \end{align}
    We now apply
    Lemma~\ref{lem:estimate_weightedmean1} and
    Lemma~\ref{lem:estimate_weightedcov}, which yields
    \begin{align}
        \Ec{\int_{\R^d}| x|^{2p}\mathrm{d}\rho^J_t( x)} &\leq \Ec{\int_{\R^d}| x|^{2p}\mathrm{d}\rho_0^{\otimes J}( x)} + \mk{4}p^2\int_0^t3\Ec {\int_{\R^d}| x|^{2p}\mathrm{d}\rho_s^J( x)} \\
        &\quad+ \bbE\Bigg[\int_{\R^d}| x|^{2(p-1)}\Big(2b_1\mk{+3} +2b_2\int_{\R^d}|\eta|^2 \mathrm{d}\rho^J_s(\eta)\Big)\mathrm{d}\rho^J_s( x)\Bigg]\\
        &\quad+\mk{ \Ec{\int_{\R^d}| x|^{2\max(0,p-2)}\Big(2b_1+2b_2\int_{\R^d}|\eta|^2\mathrm{d}\rho_s^J(\eta)\Big)\mathrm{d}\rho^J_s( x)} } \mathrm{d}s \\
        &\quad+ 2^{2p+1}p^2\Li\int_0^t\Ec{\int_{\R^d}| x|^{2(p-1)}\Big(b_1+b_2\int_{\R^d}|\eta|^2 \mathrm{d}\rho^J_s(\eta) \Big)\jz{\mathrm{d}\rho^J_s( x)}}\mathrm{d}s \\
        &\leq \Ec{\int_{\R^d}| x|^{2p}\mathrm{d}\rho_0^{\otimes J}( x)} + (\mk{4}p^2(3+2b_2)+2^{2p+\mk{2}}p^2\Li b_2)
        \int_0^t\Ec{\int_{\R^d}| x|^{2p}\mathrm{d}\rho_s^J( x)}\mathrm{d}s \\
        &\quad + (\mk{4}p^2(2b_1+3+2b_2)+2^{2p+\mk{2}}p^2\Li b_1)
        \int_0^t\Ec{\int_{\R^d}| x|^{2(p-1)}\mathrm{d}\rho^J_s( x)}\mathrm{d}s \\
        &\quad + (\mk{4}p^2(2b_1))\int_0^t\Ec{\int_{\R^d}| x|^{2\max(0,p-2)}\mathrm{d}\rho^J_s( x)}\mathrm{d}s \\
        &=: \Ec{\int_{\R^d}| x|^{2p}\mathrm{d}\rho_0^{\otimes J}( x)} + C\int_0^t\Ec{\int_{\R^d}| x|^{2p}\mathrm{d}\rho^J_s}\mathrm{d}s + \tilde C \int_0^t\Ec{\int_{\R^d}| x|^{2(p-1)}\mathrm{d}\rho^J_s( x)}\mathrm{d}s \\
        &\quad + \mk{C' \int_0^t\Ec{\int_{\R^d}| x|^{2\max(0,p-2)}\mathrm{d}\rho^J_s( x)}\mathrm{d}s}.
    \end{align}
    Here we used 
      $\int | x|^{2(p-1)} \mathrm{d}\rho( x)\int |\eta|^{2}
      \mathrm{d}\rho(\eta)\le \int | x|^{2p} \mathrm{d}\rho( x)$ \sw{which follows by Lemma~\ref{lem:auxiliary_moments} }.
      Continuing with our estimates, for $p=1$, we arrive at
    \begin{align}
        \Ec{\int_{\R^d}| x|^{2}\mathrm{d}\rho^J_t( x)} &\leq \Ec{\int_{\R^d}| x|^{2}\mathrm{d}\rho^J_0( x)} + C \int_0^t\Ec{\int_{\R^d}| x|^{2}\mathrm{d}\rho^J_s( x)}\, \mathrm{d}s + \mk{(\tilde C+C')} t.
    \end{align}
    Applying Grönwall's inequality gives the upper bound
    \begin{equation}
        \Ec{\int_{\R^d}| x|^{2}\mathrm{d}\rho^J_t( x)} \leq \left(\Ec{\int_{\R^d}| x|^{2}\mathrm{d}\rho_0^{J}( x)} + \mk{(\tilde C +C')}T\right)\exp{(CT)}\, < \infty.
    \end{equation}
    Similarly, for $p=2$, using the above estimate we obtain
    \begin{align}
        \Ec{\int_{\R^d}| x|^{4}\mathrm{d}\rho^J_t( x)} &\leq \Ec{\int_{\R^d}| x|^{4}\mathrm{d}\rho^J_0( x)} + C \int_0^t\Ec{\int_{\R^d}| x|^{4}\mathrm{d}\rho^J_s( x)}\, \mathrm{d}s + \tilde C t \sup_{t\in[0,T]}\Ec{\int_{\R^d}| x|^{2}\mathrm{d}\rho^J_t( x)} + \mk{C't}\\
        &\leq  \Ec{\int_{\R^d}| x|^{4}\mathrm{d}\rho^J_0( x)} + C \int_0^t\Ec{\int_{\R^d}| x|^{4}\mathrm{d}\rho^J_s( x)}\, \mathrm{d}s \\
        &\quad+ \tilde C T \left(\Ec{\int_{\R^d}| x|^{2}\mathrm{d}\rho_0^{J}( x)} + \mk{(\tilde C +C')} T\right)\exp{(CT)} + \mk{C'T},
    \end{align}
    such that we can derive an upper bound by applying Grönwall's inequality once more
    \begin{align}
        \Ec{\int_{\R^d}| x|^{4}\mathrm{d}\rho^J_t( x)} &\leq \Bigg( \Ec{\int_{\R^d}| x|^{4}\mathrm{d}\rho^J_0( x)} + \tilde C T \Bigg(\Ec{\int_{\R^d}| x|^{2}\mathrm{d}\rho_0^{J}( x)} \\
        &\quad + \mk{(\tilde C +C')} T\Bigg)\exp{(CT)}+\mk{C'T} \Bigg) \exp{(CT)}\, < \infty.
    \end{align}
    Using the same procedure, one may obtain a bound for the sixth
    moment.
\end{proof}

\section{Proofs of Section~\ref{sec:mfl}}
\begin{proof}[Proof of Lemma~\ref{lem:tight}]
  The proof is a straightforward adaptation of the argument in \cite[Proof of Theorem~3.3]{mfpso}, and we refer to this paper for a more detailed exposition. In particular, according to Theorem \ref{thm:chaos} it sufficies to prove the tightness of $(\mathbb P^{\theta^{1,J}})_{J\in\mathbb N}$ due to the particles'
  exchangeability. 
  We prove tightness of the solution
    $\{\theta^{1,J}\}_{J\in\mathbb{N}}$ to \eqref{eq:CBS_SDE} on
    $C([0,T];\R^d)$, using the two criteria by Aldous stated in
    Theorem~\ref{thm:aldous} in Appendix \ref{app:theorems}:

{\bf 1.} Let $\epsilon >0$, then  by Markov's inequality
    \begin{equation}
      \Pc{\mtp{\theta_t^{1,J}}>\frac{K_p(T)}{\epsilon}}
                                                            \le
                                                            \frac{\epsilon}{K_p(T)}\E{\mtp{\theta_t^{1,J}}} \leq \epsilon, \quad \forall J\in\mathbb{N},
    \end{equation}
    where $K_p(T)$ is the uniform bound on the moments from Lemma \ref{lem:momest} and $t\in[0,T]$.
    
{\bf 2.}  Let $\tau$ be a discrete $\sigma(\theta^1_s,s\in[0,T])$-stopping time with $\tau+\delta\leq T$. \jz{By \eqref{eq:CBS_SDE}} we have
    \begin{align}
        \E{\mt{\theta_{\tau+\delta}^{1,J}-\theta_\tau^{1,J}}} &= \Ec{\Big|-\int_\tau^{\tau+\delta} \theta^{1,J}_s-\M{\rho_s^J}\mathrm{d}s + \int_\tau^{\tau+\delta}\sqrt{2\Li\C{\rho_s^J}}\mathrm{d}W^{1}_s\Big|^2}\\
        &\leq {2}\Ec{\Big|\int_\tau^{\tau+\delta} \theta^{1,J}_s-\M{\rho_s^J}\mathrm{d}s\Big|^2} + {2}\Ec{\Big|\int_\tau^{\tau+\delta}\sqrt{2\Li\C{\rho_s^J}}\mathrm{d}W^{1}_s\Big|^2}
    \end{align}
    We can bound the first term via our moment estimate. Using Jensen's
    inequality
    \begin{align}
        \Ec{\Big|\int_\tau^{\tau+\delta} \M{\rho_s^J}-\theta^{1,J}_s \mathrm{d}s\Big|^2}&\le
          \sw{\delta}\int_\tau^{\tau+\delta}\E{\mt{\M{\rho_s^J}-\theta^{1,J}_s}}\mathrm{d}s \\
        &\leq 2\sw{\delta^2}  \Big(\supi{t\in[0,T]}\E{\mt{\M{\Rh}}}+\supi{t\in[0,T]}\E{\mt{\theta^{1,J}_t}}\Big) \\
        &\le 2\sw{\delta^2} K_1(T),
    \end{align}
    where for the last inequality we used Lemma \ref{lem:momest}
      and Lemma \ref{lem:estimate_weightedmean1}.
    For the second moment of the covariance of the reweighted measure we similarly obtain
    \begin{align}
      \Ec{\Big|\int_\tau^{\tau+\delta}\sqrt{\Li\C{\rho_s^{J}}} \mathrm{d}W_s\Big|^2} &=\Ec{\int_\tau^{\tau+\delta}\Big\|\sqrt{\Li\C{\rho_s^{J}}}\Big\|^2_F \mathrm{d}s} \\
        &\le \Li\int_\tau^{\tau+\delta}  b_1+b_2\Ec{\int_{\R^d}|x|^{2}\mathrm{d}\rho_s^J(x)}\mathrm{d}s \\
        &\leq \Li \delta\Bigg( b_1+b_2\supi{t\in[0,T]}\Ec{\int_{\R^d}|x|^{2}\mathrm{d}\rho_t^J(x)}\Bigg) \\
        &\leq \Li\delta (b_1 + b_2K_1(T)),
    \end{align}
    where in the first equality we used the It\^{o} Isometry, the
      first inequality uses Fubini's theorem and Lemma
      \ref{lem:estimate_weightedcov}, and the final inequality follows
      by Lemma \ref{lem:momest}.  
      In all, using \sw{$\delta\le T$} and
    Jensen's
    inequality 
    once more,
    \begin{align}
            \E{|\theta_{\tau+\delta}^{1,J}-\theta_\tau^{1,J}|} &\leq \sqrt{2\delta \cdot \max(\sw{T}K_1(T),\Li (b_1+b_2K_1(T)))} \\
            &=: \sqrt{\delta}C\,,
    \end{align}
    with $C>0$ depending on $T$ but independent of $\delta$.
    To 
    obtain
    the desired inequality \eqref{eq:aldous2}
    for given $\epsilon$, $\eta>0$,
    we define
    \begin{equation}
        \delta_0 := \min\Big(\frac{\epsilon\eta}{C^2},T\Big)
    \end{equation}
    and apply Markov's inequality to obtain
    \begin{equation}
      \underset{\delta\in[0,\delta_0]}{\sup}\mathbb{P}(|\theta^{1,J}_{\tau+\delta}-\theta^{1,J}_\tau|>\eta) \le
        \underset{\delta\in[0,\delta_0]}{\sup}\frac{1}{\eta}\E{|\theta^{1,J}_{\tau+\delta}-\theta^{1,J}_\tau|} \leq \epsilon,
    \end{equation}
    which concludes the proof.
\end{proof}

\begin{proof}[Proof of Lemma~\ref{lem:funcconvzero}]
  Plugging $\rho^J$ into $\mathbb{F}_{\varphi,t}$ defined in
    \eqref{eq:Fvarphi} we 
    get
    \begin{align}
        \label{eq:functionalempirical}
        \F{\rho^J} &= \sumj \varphi(\tjj_t) - \sumj \varphi(\tjj_0) + \int_0^t\sumj(\tjj_s-\M{\rho^J_s})\nabla\varphi(\tjj_s) \mathrm{d}s \\
        &\quad- \Li\int_0^t\sumj \sumi{i=1}{d} \sumi{k=1}{d} \C{\rho_s^J}_{ik}\frac{\partial^2}{\partial x_i\partial x_k} \varphi(\tjj_s) \mathrm{d}s \\
        &=  \sumj \Bigg( \varphi(\tjj_t) -  \varphi(\tjj_0)+ \int_0^t(\tjj_s-\M{\rho^J_s})\nabla\varphi(\tjj_s) \mathrm{d}s \\
        &\quad- \Li \int_0^t \sumi{i=1}{d} \sumi{k=1}{d} \C{\rho_s^J}_{ik}\frac{\partial^2}{\partial x_i \partial x_k} \varphi(\tjj_s) ds \Bigg).\label{eq:F_rho_j}
    \end{align}
    Since $\varphi\in{C}_c^2(\R^d)$, we may apply It\^o's formula
    to the solution $\tjj_t$ of \eqref{eq:CBS_SDE}, which yields   
    \begin{align}
        \varphi(\tjj_t) &= \varphi(\tjj_0) - \int_0^t(\tjj_s-\M{\rho^J_s})\nabla\varphi(\tjj_s) \mathrm{d}s + \int_0^t\sqrt{2\Li\C{\rho_s^J}}\nabla\varphi(\tjj_s) \mathrm{d}W^j_s \\
        &\quad+ \Li \int_0^t \sumi{i=1}{d}\sumi{k=1}{d} \C{\rho_s^J}_{ij}\frac{\partial^2}{\partial x_i\partial x_k} \varphi(\tjj_s) \mathrm{d}s.
    \end{align}
    Inserting this into \eqref{eq:F_rho_j} leads to 
    \begin{align}
        \F{\rho^J} = \sumj \int_0^t\sqrt{2\Li\C{\rho_s^J}}\nabla\varphi(\tjj_s) \mathrm{d}W^j_s\,.
    \end{align}
    For the second moment, we obtain
    \begin{align}
        \E{|\F{\rho^J}|^2} &= \frac{2}{\lambda J^2} \Ec{\Big|\sumi{j=1}{J} \int_0^t\sqrt{
        \C{\rho_s^J}
        }\nabla\varphi(\tjj_s) \mathrm{d}W_s^j\Big|^2} \\
        &= \frac{2}{\lambda J^2} \Ec{\sumi{i=1}{d}\Bigg(\sumi{j=1}{J}\int_0^t\sqrt{\C{\rho^J_s}}\nabla\varphi(\tjj_s) \mathrm{d}W^j_s\Bigg)_i^2} \\
        &\via{}{=} \frac{2}{\lambda J^2} \sumi{i=1}{d} \Ec{\Bigg(\sumi{j=1}{J}\sumi{k=1}{J}\Big(\int_0^t\sqrt{\C{\rho^J_s}}\nabla\varphi(\tjj_s) \mathrm{d}W^j_s\Big)\Big(\int_0^t\sqrt{\C{\rho^J_s}}\nabla\varphi(\theta^{k,J}_s) \mathrm{d}W^k_s\Big)\Bigg)_i} \\
        &\via{}{=} \frac{2}{\lambda J^2} \sumi{i=1}{d} \sumi{j=1}{J}\Ec{\int_0^t\nabla\varphi(\tjj_s)^\top\C{\rho^J_s}\nabla\varphi(\tjj_s) \mathrm{d}s} \\
        &\leq C(\lambda, K_p(T), T, \norm[L^\infty]{\nabla\varphi}\frac{1}{J} \overset{J\to\infty}{\to} 0
    \end{align}
    where for the final inequality we used
    Lemma \ref{lem:momest} and Lemma \ref{lem:estimate_weightedcov},
    and the 
    equality in the second to last line
    holds due to It\^{o}'s Isometry and the fact that $\E{\int_0^tg(t, \theta_t)d(B_s^j\otimes B_s^k)}=0$ for $j\neq k$ and $g\in\jz{\L{2}}$;
    this holds 
    due to $\varphi\in C_c^2(\R^d)$ and Lemma \ref{lem:estimate_weightedcov}.
    We were able to apply Lemma \ref{lem:momest} due to our assumption about the initialisation.
    \end{proof}

\begin{proof}[Proof of Lemma~\ref{lem:funcconvlimit}]
  \sw{Firstly, we observe that} \mk{for every $t\in[0,T], \{\rho^J_t\}_{J\in\mathbb N}$ converges\footnote{remember that we identify the subsequence $\{\rho^{J_k}\}_{k\in\mathbb N}$ as $\{\rho^J\}_{J\in\mathbb N}$.} $\tilde{\mathbb P}_t$-almost surely (on $(\tilde \Omega_t,\tilde{\mathcal F}_t,\tilde{\mathbb P}_t)$) to $\rho_t\in\mathcal{P}(\R^{D})$.}
  
  It suffices to check the $\L{1}$-convergence of the
  terms of $\F{\overline\rho^J}$ in \eqref{eq:F_rho_j} one by one.
    The difference to \cite[Theorem 3.3]{huangcbo} lies in the diffusion term, which we split as
    \begin{align}
        &\Bigg|\int_0^t\Li\sumi{i=1}{d}\sumi{k=1}{d} \int_{\R^d}\C{\rho_s^J}_{ik}\frac{\partial^2}{\partial x_i\partial x_k} \varphi( x) \mathrm{d}\rho_s^J( x)\mathrm{d}s-\int_0^t\Li\sumi{i=1}{d}\sumi{k=1}{d}\int_{\R^d} \C{\rho_s}_{ik}\frac{\partial^2}{\partial x_i\partial x_k} \varphi( x) \mathrm{d}\rho_s( x) \mathrm{d}s\Bigg| \\
        &\qquad\leq \Li \Bigg(\int_0^t\Big|\sumi{i=1}{d}\sumi{k=1}{d} \int_{\R^d}\C{\rho_s^J}_{ik}\frac{\partial^2}{\partial x_i\partial x_k} \varphi( x) \mathrm{d}(\rho_s^J( x)-\rho_s( x))\Big|\mathrm{d}s \\
        &\qquad\quad+ \int_0^t\Big|\sumi{i=1}{d}\sumi{k=1}{d} \int_{\R^d}(\C{\rho^J_s}_{ik}-\C{\rho_s}_{ik})\frac{\partial^2}{\partial x_i\partial x_k} \varphi( x) \mathrm{d}\rho_s( x)\Big|\mathrm{d}s\Bigg) \\
        \label{eq:bothintegrands}
         &\qquad=: \Li\left(\int_0^t|I_1^J(s)|\mathrm{d}s + \int_0^t|I_2^J(s)|\mathrm{d}s\right),
    \end{align}
    where we used the triangle inequality and Jensen's inequality.
    The expectation of the first integrand vanishes, since $\varphi$
    vanishes at the boundaries: using Jensen's inequality, Lemma
    \ref{lem:estimate_weightedcov} and $K_1(T)$ from Lemma \ref{lem:momest}
    \begin{align}
      \E{|I^J_1(s)|} &\le \sumi{i=1}{d}\sumi{k=1}{d}\Ec{\Big|\int_{\R^d}\C{\rho_s^J}_{ik} \frac{\partial^2}{\partial x_i\partial x_k} \varphi( x) \mathrm{d}(\rho_s^J( x)-\rho_s( x))\Big|} \\
        &\leq \left(K_1(T)^{\frac{1}{2}}\sumi{i=1}{d}\sumi{k=1}{d}\Ec{\Big|\int_{\R^d}\frac{\partial^2}{\partial x_i\partial x_k}\varphi( x)\mathrm{d}(\rho_s^J( x)-\rho_s( x))\Big|}\right)\to 0\qquad\text{as }J\to\infty.\label{eq:integrandone}
    \end{align}
    Moreover, for the second integrand, we claim that for all $s\in[0,T]$
    \begin{align}\label{eq:integrandtwo}
        \E{|I^J_2(s)|} &= \Ec{\Bigg|\sumi{i=1}{d}\sumi{k=1}{d}\int_{\R^d}(\C{\rho^J_s}_{ik}-\C{\rho_s}_{ik})
        \frac{\partial^2}{\partial x_i\partial x_k} \varphi( x)\mathrm{d}\rho_s( x) \Bigg|}\nonumber \\
        &\leq \sup_{ x\in\R^d}\norm[F]{\nabla^2 \varphi( x)}\cdot\E{\|\C{\rho^J_s}-\C{\rho_s}\|_F}
        \to 0\qquad\text{as }J\to\infty.
    \end{align}
    To show this
      claim, we first note that there holds the
      \sw{$\tilde{\mathbb P}_s$-almost sure} convergence
    \begin{equation}
        \underset{J\to\infty}{\lim} \|\C{\rho^J_s}-\C{\rho_s}\|_F = 0, 
      \end{equation}
      for each fixed $s$ because $ x\exp({-\beta
        f( x)})$ and $\exp{(-\beta f( x))}$ belong to
      ${C}_b(\R^d)$. 
        \sw{More precisely}, together with  
        \sw{$\tilde{\mathbb P}_t$-almost sure} convergence
      of the $\rho^J_t$ towards
        $\rho_t$, 
        this implies
    \begin{align}
        \underset{J\to\infty}{\lim} \C{\rho_t^J} &= \underset{J\to\infty}{\lim} \int_{\R^d}( x-\mathcal{M}_\beta(L_\beta\rho^J_t))\otimes( x-\mathcal{M}_\beta(L_\beta\rho_t^J)\,\mathrm{d}L_\beta \rho_t^J( x) \\
        &= \underset{J\to\infty}{\lim} \int_{\R^d}\left( x-\frac{\i{ x\exp({-\beta f( x))}}{\mathrm{d}\rho_t^J( x)}}{\i{\exp{(-\beta f( x))}}{\mathrm{d}\rho_t^J( x)}}\right)\otimes\left( x-\frac{\i{ x\exp({-\beta f( x))}}{\mathrm{d}\rho_t^J( x)}}{\i{\exp{(-\beta f( x))}}{\mathrm{d}\rho_t^J( x)}}\right)\\
      &\quad\cdot \frac{\exp{(-\beta f( x))}}{{\i{\exp{(-\beta f( x))}}{\mathrm{d}\rho_t^J( x)}}}\,\mathrm{d}\rho_t^J( x) \\
        &= \int_{\R^d}\left( x-\frac{\i{ x\exp({-\beta f( x))}}{\mathrm{d}\rho_t( x)}}{\i{\exp{(-\beta f( x))}}{\mathrm{d}\rho_t( x)}}\right)\otimes\left( x-\frac{\i{ x\exp({-\beta f( x))}}{\mathrm{d}\rho_t( x)}}{\i{\exp{(-\beta f( x))}}{\mathrm{d}\rho_t( x)}}\right)\\
      &\quad\cdot\frac{\exp{(-\beta f( x))}}{{\i{\exp{(-\beta f( x))}}{\mathrm{d}\rho_t( x)}}}\,\mathrm{d}\rho_t( x) \\
        &= \int_{\R^d}( x-\mathcal{M}(L_\beta\rho_t))\otimes( x-\mathcal{M}(L_\beta\rho_t)\,\mathrm{d}L_\beta \rho_t( x) \\
        \label{eq:covariance}
        &= \C{\rho_t}\,,
    \end{align}
    \sw{$\tilde{\mathbb P}_t$-almost surely}. Here we wrote integrals over $\R^d$ with respect to a measure $\mu$ as $\i{\cdot}{\mathrm{d}\mu( x)}$ for better legibility. \mk{Importantly, we were able to use weak convergence of the time marginals $\rho^J_t \to \rho_t$ according to \cite[Lemma 2.3 (2)]{huangcbo}.}
    \sw{We apply Lemma~\ref{lem:estimate_weightedcov} to verify that 
    $\lim_{J\to\infty} \E{\|\C{\rho^J_s}-\C{\rho_s}\|_F} = 0$, which can be seen as follows. Let $A>0$ be arbitrary, and consider
    \begin{align*}
        &\E{ \|\C{\rho^J_s}-\C{\rho_s}\|_F }\\ &= \E{ \|\C{\rho^J_s}-\C{\rho_s}\|_F \mathds{1}_{\|\C{\rho^J_s}-\C{\rho_s}\|_F \le A} } + \E{ \|\C{\rho^J_s}-\C{\rho_s}\|_F \mathds{1}_{\|\C{\rho^J_s}-\C{\rho_s}\|_F>A} }\\
        & \le \E{ \|\C{\rho^J_s}-\C{\rho_s}\|_F \mathds{1}_{\|\C{\rho^J_s}-\C{\rho_s}\|_F \le A} } + \E{ \|\C{\rho^J_s}-\C{\rho_s}\|_F^2 }^{\frac12} \tilde{\mathbb{P}}(\|\C{\rho^J_s}-\C{\rho_s}\|_F>A)^{\frac12} \\
        &\le \E{ \|\C{\rho^J_s}-\C{\rho_s}\|_F \mathds{1}_{\|\C{\rho^J_s}-\C{\rho_s}\|_F \le A} } + \frac{\E{ \|\C{\rho^J_s}-\C{\rho_s}\|_F^2 }}{A^2}\,,
    \end{align*}
    where we have first applied Hölder's inequality followed by Markov's inequality. Note that by Lemma~\ref{lem:estimate_weightedcov} and Lemma~\ref{lem:momest},
    \[ \E{ \|\C{\rho^J_s}-\C{\rho_s}\|_F^2} \le 2 \E{ \|\C{\rho^J_s}\|_F^2} + 2 \E{ \|\C{\rho_s}\|_F^2} \le K_2(T)<\infty\]
    uniformly in $J$. Taking the limit $J\to\infty$ yields
    \[0\le \limsup_{J\to\infty}\E{ \|\C{\rho^J_s}-\C{\rho_s}\|_F } \le \frac{K_2(T)}{A^2}\, ,  \]
    and since $A>0$ is arbitrary, we also obtain $\lim_{J\to\infty} \E{\|\C{\rho^J_s}-\C{\rho_s}\|_F} = 0$.
    }
    \sw{In summary, we verified the convergence of the diffusion term}
    \begin{align}\label{eq:summand3}
      \underset{J\to\infty}{\lim}& \mathbb{E}\Bigg[\Li \Big|\int_0^t\sumi{i=1}{d}\sumi{k=1}{d} \int_{\R^d}\C{\rho_s^J}_{ik}\frac{\partial^2}{\partial x_i\partial x_k} \varphi( x)\mathrm{d}\rho_s^J( x)\mathrm{d}s\\
        &\qquad -\int_0^t\sumi{i=1}{d}\sumi{k=1}{d} \int_{\R^d}\C{\rho_s}_{ik}\frac{\partial^2}{\partial x_i\partial x_k} \varphi( x) \mathrm{d}\rho_s( x) \mathrm{d}s\Big|\Bigg] = 0.
    \end{align}
    The remaining two summands of the functional are identical to CBO, and their convergence has already been obtained in \cite[Theorem 3.3]{huangcbo}, i.e.\ 
    it holds
    \begin{equation}
        \label{eq:summand1}
        \underset{J\to\infty}{\lim} \Ec{\big|\left(\int_{\R^d}\varphi( x)\mathrm{d}\rho^J_t( x)-\int_{\R^d}\varphi( x)\mathrm{d}\rho^J_0( x)\right)-\left(\int_{\R^d}\varphi( x) \mathrm{d}\rho_t( x)-\int_{\R^d}\varphi( x)\mathrm{d}\rho_0( x)\right)\big|} = 0
      \end{equation}
      and
      \begin{equation}
        \label{eq:summand2} \underset{J\to\infty}{\lim} \Ec{\Big|\int_0^t \int_{\R^d}( x-\M{\rho^J_s})\cdot\nabla\varphi( x)\mathrm{d}\rho^J_s( x)\mathrm{d}s-\int_0^t\int_{\R^d}( x-\M{\rho_s})\cdot\nabla\varphi( x)\mathrm{d}\rho_s( x)\mathrm{d}s\Big|} = 0. 
    \end{equation}
    Equations \eqref{eq:summand3}, \eqref{eq:summand1}, and 
    \eqref{eq:summand2} together
    with
    the triangle inequality
    give
    \begin{equation}
        \underset{J\to\infty}{\lim} \E{|\F{\overline\rho^J}-\F{\overline\rho}|} = 0,
    \end{equation}
    which is the desired result.
\end{proof}

\begin{proof}[Proof of Theorem~\ref{thm:weakly_conv_fpe}] The proof is a direct adaptation of \cite[Theorem 1.1]{huangcbo}.
  \mk{Considering the measure flow $\overline\rho=(\rho_t)_{t\in[0,T]}$ of time marginals of the limit $\rho$, we must check the two conditions in the definition of a weak solution. We begin with the first point, which requires checking that $t\mapsto \rho_t$ is} continuous in time in the sense
  of \eqref{eq:time_continuity_condition}: For $\phi\in{C}_b(\R^d)$
  and $t_n\to t$, dominated convergence yields
    \begin{equation*}
      \int_{C([0,T]; \R^d)}\phi({f}(t_n))\mathrm{d}\rho({f}) \overset{n\to\infty}{\to} \int_{C([0,T]; \R^d)} \phi({f(t)}) \mathrm{d}\rho({f}).
    \end{equation*}
    With the definition of the evaluation map $e_t:C([0,T], \mathbb R^d)\to \mathbb R^d, f\mapsto f(t)$ and the time marginals $\rho_t:=(e_t)_\#\rho=\rho(e_t^{-1}(\cdot))$ one can equivalently write the above as  
    \begin{equation*}
      \int_{\R^d}\phi(x)\mathrm{d}\rho_{t_n}(x) \overset{n\to\infty}{\to} \int_{\R^d}\phi(x)\mathrm{d}\rho_t(x)
    \end{equation*}
    using change of variables, which is applicable because $\phi$ is measurable. 
    
    The second part of the Definition now follows straightforwardly by combining the previously obtained convergence results from Lemma \ref{lem:funcconvzero} and Lemma \ref{lem:funcconvlimit}:
    \begin{align}
      |\F{\overline\rho}| &= \E{|\F{\overline\rho}|} \\
        &= \E{|\F{\overline\rho}-\F{\overline\rho^J}+\F{\overline\rho^J}|} \\
        &\leq \E{|\F{\overline\rho}-\F{\overline\rho^J}|} + \E{|\F{\overline\rho^J}|} \\
        &\leq \E{|\F{\overline\rho}-\F{\overline\rho^J}|} + \frac{C}{\sqrt{J}} \to 0\qquad\text{as }J\to\infty.
    \end{align}
\end{proof}

  We next recall a standard argument to bound the difference
  of the square root of two matrices.
  
\begin{lemma}\label{lem:auxiliary_sqrt_cov_fixed}
  Let $A\in\R^{d\times d}$ be symmetric positive semi-definite and let
  $B\in\R^{d\times d}$ be symmetric positive definite. Then
  \begin{equation*}
    \|\sqrt{A}-\sqrt{B}\|_2\le 2\sqrt{\|B^{-1}\|_2}\|A-B\|_2.
  \end{equation*}
\end{lemma}
\begin{proof}
  Let $a\ge 0$. Substituting $t=\sqrt{a}\tan(s)$ one can show
  $\frac{2}{\pi}\int_0^\infty \frac{a}{a+t^2}\mathrm{d}
  t=\frac{2}{\pi}\int_{0}^{\pi/2}\sqrt{a} \mathrm{d} s = \sqrt{a}$. Using the
  singular value decomposition (SVD) and applying the above equality to the
  singular values yields the well-known identity
  \begin{equation*}
    \sqrt{A} = \frac{2}{\pi}\int_0^\infty (A+t^2I)^{-1}A \mathrm{d}t,
  \end{equation*}
  where $I\in\R^{d\times d}$ is the identity matrix. Since the same is true for $B$, we find
  \begin{align*}
    \sqrt{A}-\sqrt{B} &= \frac{2}{\pi}\int_0^\infty (A+t^2 I)^{-1}A - (B+t^{2} I)^{-1}B \mathrm{d}t\\
    &=\frac{2}{\pi}\int_0^\infty ((A+t^2 I)^{-1}-(B+t^2 I)^{-1})A \mathrm{d}t +\frac{2}{\pi}\int_0^\infty (B+t^{2} I)^{-1}(A-B) \mathrm{d}t\\
    &=:E+F.
  \end{align*}

  We first bound $\|F\|_2$. Denote by $\sigma=\|B^{-1}\|_2^{-1}>0$ the
  smallest eigenvalue of $B$. The smallest eigenvalue of $B+t^2I$ is
  then $\sigma+t^2$. Thus
  \begin{equation}\label{eq:BptI}
    \|(B+t^{2}I)^{-1}\|_2=\frac{1}{\sigma+t^2}
  \end{equation}
    and
  \begin{equation*}
    \|F\|_2\le \|A-B\|_2\frac{2}{\pi}\int_0^\infty \frac{1}{\sigma+t^2} \mathrm{d} t = \|A-B\|_2\frac{1}{\sqrt{\sigma}}.
  \end{equation*}
  
  To bound $\|E\|_2$, we use the matrix identity $X^{-1}-Y^{-1}=Y^{-1}(Y-X)X^{-1}$
  to obtain
  \begin{align*}
    \|((A+t^2 I)^{-1}-(B+t^2 I)^{-1})A\|_2&\le
    \|(B+t^2 I)^{-1}\|_2\|A-B\|_2\|(A+t^2 I)^{-1}A\|_2\\
    &\le \frac{1}{\sigma+t^2}\|A-B\|_2.
  \end{align*}
  Here we used \eqref{eq:BptI} and the fact that
  $\|(A+t^2 I)^{-1}A\|_2\le 1$, which can be seen using the SVD of
  $A$. As before $\|E\|_2\le \|A-B\|_2/\sqrt{\sigma}$, which concludes
  the proof.
\end{proof}

\begin{proof}[Proof of Lemma~\ref{lem:substitutestability}]
The proof of the first claim is given in \cite[Lemma~3.2]{cboanalysis}. 

For the second claim, let \sw{$\mu$, $\nu\in \mathcal P_6(\R^d)$} with 
\[\sw{\left(\int_{\R^d} | x|^6 \mathrm{d}\mu( x)\right)^{\frac16},~\left(\int_{\R^d} |\widehat x|^6 \mathrm{d}\nu(\widehat x)\right)^{\frac16}\le K}. \]
We define the normalizing constants with respect to the weight function $w_\beta^f$ by
\[ Z_\mu := \int_{\R^d} w_\beta^f( x)\, \mathrm{d}\mu( x)\quad \text{and}\quad Z_\nu := \int_{\R^d} w_\beta^f(\widehat x)\, \mathrm{d}\nu(\mathrm{d} \widehat x).\]
According to
\cite[Lemma~3.1]{cboanalysis}, there exists a constant $c_K>0$ such that
\[ \frac{w_\beta^f( x)}{Z_\mu},~\frac{w_\beta^f( x)}{Z_\nu} \le c_K \]
for all $ x\in\R^d$. By definition of the weighted covariance matrix
\begin{align*}
\C{\mu}-\C{\nu} &= \int_{\R^d} \frac{w_\beta^f( x)}{Z_\mu} x  x^\top \mathrm{d}\mu( x) - \int_{\R^d} \frac{w_\beta^f(\widehat x)}{Z_{\nu}}\widehat x \widehat x^\top \mathrm{d}\nu(\widehat x)+ M_\beta(\mu) M_\beta(\mu)^\top - M_\beta(\nu)M_\beta(\nu)^\top\,.
\end{align*}
Next, we estimate
\begin{align*} 
\|\C{\mu}-\C{\nu}\|_F \le \normc[F]{\int_{\R^d} \frac{w_\beta^f( x)}{Z_\mu} x  x^\top \mathrm{d}\mu( x) - \int_{\R^d} \frac{w_\beta^f(\widehat x)}{Z_{\nu}}\widehat x \widehat x^\top \mathrm{d}\nu(\widehat x)} + \|M_\beta(\mu) M_\beta(\mu)^\top - M_\beta(\nu)M_\beta(\nu)^\top\|_F.
\end{align*}
We will make use of the following estimates for arbitrary vectors $x,y\in\R^d$
\begin{equation} \label{eq:fundamental_estimate}
\| xx^\top -yy^\top \|_F^2 =|y|^2 |x-y|^2+|x|^2 |x-y|^2 + 2\langle y,x-y\rangle \langle x,y-x\rangle \le (|x|+|y|)^2|x-y|^2 
\end{equation}
where we have applied the Cauchy-Schwarz inequality. Hence, we first observe that
\[\|M_\beta(\mu) M_\beta(\mu)^\top - M_\beta(\nu)M_\beta(\nu)^\top\|_F \le (|M_\beta(\mu)|+|M_\beta(\nu)|)|M_\beta(\mu)-M_\beta(\nu)| \le 2(b_1 +b_2 c_K) c_0 W_2(\mu,\nu), \]
where we have used \cite[Lemma~3.2]{cboanalysis} and Lemma~\ref{lem:estimate_weightedmean1}. Secondly, let $\pi$ be an arbitrary coupling of $\mu$ and $\nu$ such that we can write
\begin{align*}
\normc[F]{\int_{\R^d} \frac{w_\beta^f( x)}{Z_\mu} x  x^\top \mathrm{d}\mu( x) - \int_{\R^d} \frac{w_\beta^f(\widehat x)}{Z_{\nu}}\widehat x \widehat x^\top \mathrm{d}\nu(\widehat x)} &= \normc[F]{\int_{\R^d}\int_{\R^d} \left(\frac{w_\beta^f( x)}{Z_\mu} x  x^\top  -  \frac{w_\beta^f(\widehat x)}{Z_{\nu}}\widehat x \widehat x^\top \right) \mathrm{d}\pi( x,\widehat x)}\\
&\le \int_{\R^d}\int_{\R^d} \normc[F]{\frac{w_\beta^f( x)}{Z_\mu} x  x^\top  -  \frac{w_\beta^f(\widehat x)}{Z_{\nu}}\widehat x \widehat x^\top} \mathrm{d}\pi( x,\widehat x)\,.
\end{align*}
The outline of the remaining proof follows similar steps as the proof of \cite[Lemma~3.2]{cboanalysis}. In particular, we use triangle inequality to bound the integrand by
\begin{align}
    \normc[F]{\frac{w_\beta^f( x)}{Z_\mu} x  x^\top  -  \frac{w_\beta^f(\widehat x)}{Z_{\nu}}\widehat x \widehat x^\top} &\le \| x  x^\top - \widehat  x\widehat x^\top \|_F \left|\frac{w_\beta^f(\widehat x)}{Z_\mu} \right| \label{eq:Wbound1}\\ &\quad+ \| x x^\top \|_F \frac{|w_\beta^f( x)-w_\beta^f(\widehat x)|}{Z_\mu}\label{eq:Wbound2}\\ &\quad + \|\widehat x\widehat x^\top\|_F |w_\beta(\widehat x)|  \frac{|\int_{\R^d}\int_{\R^d} w_\beta^f(y)-w_\beta^f(\widehat y) d\pi(y,\widehat y)|}{Z_\mu Z_\nu}\, ,\label{eq:Wbound3}
\end{align}
where we used the equality
\[\frac{w_\beta^f( x)}{Z_\mu} x  x^\top  -  \frac{w_\beta^f(\widehat x)}{Z_{\nu}}\widehat x \widehat x^\top = ( x x^\top - \widehat  x \widehat x^\top ) \frac{w_\beta^f(\widehat  x)}{Z_\mu} +  x x^\top \frac{w_\beta^f( x)-w_\beta^f(\widehat  x)}{Z_\mu} + \widehat  x \widehat  x^\top w_\beta^f(\widehat x) \left(\frac{1}{Z_\mu} - \frac{1}{Z_\nu} \right)\]
and 
\[\frac{1}{Z_\mu} - \frac{1}{Z_\nu} = \frac{\int_{\R^d}\int_{\R^d} w_\beta^f(y)-w_\beta^f(\widehat y) \mathrm{d}\pi(y,\widehat y)}{Z_\mu Z_\nu}\,. \]
We will integrate over and bound each tearm in \eqref{eq:Wbound1}--\eqref{eq:Wbound3} separately. We start with \eqref{eq:Wbound1} and apply \eqref{eq:fundamental_estimate} which yields
\begin{align*} 
\int_{\R^d}\int_{\R^d} \| x  x^\top - \widehat  x x^\top \|_F \left|\frac{w_\beta^f(\widehat x)}{Z_\mu} \right| \mathrm{d}\pi( x,\widehat x) &\le c_K \int_{\R^d}\int_{\R^d} (| x|+|\widehat x|) | x-\widehat x|\mathrm{d}\pi( x,\widehat x)\\
&\le c_K \left(\int_{\R^d}\int_{\R^d} (| x|+|\widehat x|)^2\mathrm{d}\pi( x,\widehat x)\right)^{\frac12} \left(\int_{\R^d}\int_{\R^d} (| x-\widehat x|)^2\mathrm{d}\pi( x,\widehat x)\right)^{\frac12} \\
&\le 4 c_K K \left(\int_{\R^d}\int_{\R^d} (| x-\widehat x|)^2\mathrm{d}\pi( x,\widehat x)\right)^{\frac12}\,,
\end{align*}
where we have applied Hölder's inequality in the second line. For the second term \eqref{eq:Wbound2} we apply Assumption~\eqref{asp:costfuns} to derive
\begin{align*}  
\int_{\R^d}\int_{\R^d} \| x x^\top \|_F \frac{|w_\beta^f( x)-w_\beta^f(\widehat x)|}{Z_\mu} \mathrm{d}\pi( x,\widehat x) &\le 2c_K\int_{\R^d}\int_{\R^d} \| x x^\top \|_F (| x|+|\widehat x|)| x-\widehat  x| \mathrm{d}\pi( x,\widehat x)\\
&= 2c_K\int_{\R^d}\int_{\R^d} (| x|^3+| x|^2|\widehat x|) | x-\widehat  x| \mathrm{d}\pi( x,\widehat x)\\
&\le 2c_K \left( \int_{\R^d}\int_{\R^d} (| x|^3+| x|^2|\widehat x|)^2 \mathrm{d}\pi( x, \widehat  x) \right)^{\frac12} \left( \int_{\R^d}\int_{\R^d} | x-\widehat x|^2 \mathrm{d}\pi( x, \widehat  x) \right)^{\frac12} \\
&\le 2c_K \tilde K \left( \int_{\R^d}\int_{\R^d} | x-\widehat x|^2 \mathrm{d}\pi( x, \widehat  x) \right)^{\frac12},
\end{align*}
where we once again used Hölder's inequality in the third line and afterwards 
\sw{the assumption on bounded second moments.}
Finally, applying Hölder's inequality once more we estimate \eqref{eq:Wbound3} by
\begin{align*}
    \int_{\R^d}\int_{\R^d}&\|\widehat x\widehat x^\top\|_F |w_\beta(\widehat x)|  \frac{|\int_{\R^d}\int_{\R^d} w_\beta^f(y)-w_\beta^f(\widehat y) \mathrm{d}\pi(y,\widehat y)|}{Z_\mu Z_\nu} \mathrm{d}\pi( x,\widehat x)\\ &\le c_K^2 \beta \text{Lip}(f) \int_{\R^d}\int_{\R^d} \|\widehat  x\widehat x^\top\|_F \mathrm{d}\pi( x,\widehat x) \int_{\R^d}\int_{\R^d} (|y|+|\widehat y|) |y-\widehat y| \mathrm{d}\pi(y,\widehat y)\\
    &\le c_K^2 \beta \text{Lip}(f) 4K\int_{\R^d}\int_{\R^d} | x|^2 \mathrm{d}\pi( x,\widehat x) \left(\int_{\R^d}\int_{\R^d} |y-\widehat y|^2 \mathrm{d}\pi(y,\widehat y)\right)^{\frac12} \\
    &\le c_K^2 \beta \text{Lip}(f) 4 K^2 \left(\int_{\R^d}\int_{\R^d} |y-\widehat y|^2 \mathrm{d}\pi(y,\widehat y)\right)^{\frac12}\,.
\end{align*}
Taking the infimum over all couplings of $\mu$ and $\nu$ yields the claim. 
\end{proof}

\begin{proof}[Proof of Theorem~\ref{thm:mckeanexistence}]
We begin by linearizing the process. Consider for fixed $\rho_0\in\mathcal{P}_4(\R^d)$ and some given $(u, D)\in C([0,T],\R^d\times\R^{d\times d})$ the following linear SDE
\begin{align}\label{eq:linearizedmckean}
    \mathrm{d}Y_t &= -(Y_t-u_t) \mathrm{d}t + \sqrt{2\Li} D_t \mathrm{d}W_t,\quad\text{where}\quad \mathbb{P}^{Y_0}=\rho_0.
\end{align}
By standard SDE theory, e.g. \cite[Thm. 5.2.1]{oksendal}, this SDE has a unique solution with laws $\nu_t=\mathbb{P}^{Y_t}, t\geq0$. More precisely, for every $t$, $\nu_t$ is the law of $Y_t$, leading to the function $\nu \in C([0,T],\mathcal{P}(\R^d))$.

We then define a mapping from $(u, D)$ to the continuous function which consists of the regularized mean and regularized covariance of $\nu_t$ at every timestep:
\begin{equation}
  \mathcal{T}:=
  \begin{cases}
  C([0,T],\R^d\times\R^{d\times d})\to C([0,T],\R^d\times \R^{d\times d})\\
    (u,D) \mapsto 
    \big(\M{\nu_t},\sqrt{\C{\nu_t}}\big)_{t\in[0,T]},
    \end{cases}
\end{equation}
for which we prove the conditions of Schaefer's Fixed Point Theorem \eqref{thm:schaefersfpt}. The well-definedness of the mapping follows firstly from the ranges of $\mathcal{M}_\beta$ and $\mathcal{C}_\beta$, and secondly from their respective (Hölder) continuitiy, which we prove below.
Theorem \ref{thm:schaefersfpt} is sufficient to go from the above linearized SDE \eqref{eq:linearizedmckean} to the nonlinear McKean process as a fixed point of the mapping $\mathcal{T}$ ensures that $(\mathcal{M}_\beta(\nu_t),\sqrt{\mathcal{C}_\beta(\nu_t)})$ can be plugged in for $(u_t, D_t)$, which yields \sw{a solution of} the nonlinear SDE in \eqref{eq:mckeansde}. 

We now begin with checking the compactness requirement of Theorem \eqref{thm:schaefersfpt}, which we obtain by showing the map $t\mapsto (\M{\nu_t},\sqrt{\C{\nu_t}})$ is Hölder continuous and making use of the compact embedding from Hölder continuous functions to continuous functions.

\sw{By \cite[Thm. 7.1.2]{arnold}}, there exists $c'>0$ s.t. 
\begin{align}
    \E{|Y_t|^6} &\leq (1+\E{|Y_0|^6})\exp{(c't)},
\end{align}
for all $t\in[0,T]$, i.e.\ $\sup_{t\in[0,T]} \int|x|^6d\nu_t\leq K$ for a $K<\infty$, which means we will be able to apply Lemma \ref{lem:substitutestability}. 
Furthermore, we have for $t>s$, $t$, $s \in (0,T)$
\begin{align*}
    \E{|Y_t-Y_s|^2} &= \Ec{\Big(-\int_s^tY_r-u_r \mathrm{d}r\Big)^2}+\Ec{\Big(\int_s^t\sqrt{2\Li} D_r \mathrm{d}W_r\Big)^2}\\ &\quad+2\sqrt{2\Li}\Ec{\Big(\int_s^t(Y_r-u_r) \mathrm{d}r\Big)\Big(\int_s^tD_r \mathrm{d}W_r\Big)}.
\end{align*}
Bounding the three terms respectively using for example the Cauchy-Schwarz inequality, the It\^{o} Isometry and the fact that the third term is a martingale, we have 
\begin{align}
    \E{|Y_t-Y_s|^2} &{\leq} |t-s|\Ec{\int_s^t|Y_r-u_r|^2 \mathrm{d}r}+2\Li\Ec{\int_s^tD_r^2\mathrm{d}r}+ 0 \\
    &\leq 2(T+2\Li)(K+\|u\|_\infty^2+\|D^2\|^2_\infty)|t-s|\\
    &=: c|t-s|.
\end{align}
Applying \cite[Lemma 3.2]{cboanalysis} with $\mu=\nu_t$ and $\hat{\mu}=\nu_s$, we arrive at 
\begin{equation}
  |\M{\nu_t}-\M{\nu_s}| \le c_0 W_2(\nu_t,\nu_s) \leq c_0c^{\frac{1}{2}} |t-s|^{\frac{1}{2}}.
\end{equation}
\sw{for $c_0>0$ from Lemma~\ref{lem:substitutestability}.}
Similarly, using the well-known Powers-Størmer inequality \cite{powers1970free}, estimating the resulting Schatten $1$-norm $|[\cdot]|_1$ using the Frobenius norm and applying Lemma \ref{lem:substitutestability} we obtain
\begin{align}
  \|\sqrt{\C{\nu_t}}-\sqrt{\C{\nu_s}}\|_F &\le (|[\C{\nu_t}-\C{\nu_s}]|_1)^{\frac12} \\
  &\le (d^{\frac12}\|\C{\nu_t}-\C{\nu_s}\|_F)^{\frac12} \\
  &\le  d^{\frac14} \sqrt{c_1 W_2(\nu_t,\nu_s)} \\
  &\leq c_1^{\frac12}d^{\frac14}  |t-s|^{\frac{1}{4}}.
\end{align}
This yields the Hölder continuity with exponent  
$\frac{1}4$ of 
\begin{equation}
  f:=\begin{cases}
    [0,T] \to \R^d\times\R^{d\times d} \\
    t \mapsto (\M{\nu_t},\sqrt{\C{\nu_t}}).
    \end{cases}
\end{equation}
The compact embedding 
\begin{align}
    {C}^{0,\frac{1}4}([0,T],\R^d\times\R^{d\times d}) \hookrightarrow C([0,T],\R^d\times\R^{d\times d})
\end{align}
therefore provides the compactness of $\mathcal{T}$ (See \cite[Thm. 10.6]{alt}.).

Computing the second moment using It\^{o}'s formula, we get
\begin{equation}
    \label{eq:timederivativeempirical}
    \frac{d}{dt}\int_{\mathbb{R}^d}| x|^2 \mathrm{d}\rho_t = \int_{\mathbb{R}^d}\Big[2( x-u_t)\cdot  x+\Li\Tr{D_t^2}\Big] \mathrm{d}\rho_t( x).
\end{equation}
Lemma \ref{lem:estimate_weightedcov} then yields the following for the Diffusion term:
\begin{equation*}
        \Tr{D_t^2} = \tau^2\Tr{\C{\rho_t}} \leq  \tau^2 \left(b_1 + b_2 \int_{\R^d}| x|^2 \mathrm{d}\rho_t( x)\right). 
\end{equation*}
Similarly, by Lemma \ref{lem:estimate_weightedmean1} it holds that 
\begin{equation*}
    |u_t|^2 = \tau^2\M{\rho_t}^2  \le
      \tau^2\int_{\R^d}| x|^2 \mathrm{d}L_\beta \rho_t( x) \leq \tau^2\left(b_1 + b_2 \int_{\R^d}| x|^2 \mathrm{d}\rho_t( x)\right).
\end{equation*}
Using Cauchy-Schwarz and the arithmetic-geometric mean inequality, we furthermore obtain
\begin{align}
    -2\int_{\R^d} x\cdot u_t \mathrm{d}\rho_t( x) &\le 2 \int | x\cdot u_t|\mathrm{d}\rho_t( x) \\
    &\le 2\int | x||u_t|\mathrm{d}\rho_t( x) \\
    &\le \int | x|^2 + |u_t|^2 \mathrm{d}\rho_t( x) \\
    &\le \int | x|^2 \mathrm{d}\rho_t( x) + |u_t|^2.
\end{align}
Therefore, we can bound \eqref{eq:timederivativeempirical} as follows:
\begin{align}
    \frac{d}{dt}\int_{\R^d}| x|^2 \mathrm{d}\rho_t( x) &\leq \tau^2b_1\Li +(\tau^2b_2\Li+2) \int_{\mathbb{R}^d} | x|^2 \mathrm{d}\rho_t( x) + |u_t|^2 \\
    &\leq \tau^2b_1(\Li+1) +(\tau^2 b_2(\Li+1)+2)\int_{\R^d}| x|^2 \mathrm{d}\rho_t( x).
\end{align}
Applying Grönwall's inequality yields
\begin{align}
\label{eq:sfpteigenvectorbound}
    \int_{\R^d}| x|^2 \mathrm{d}\rho_t( x) &\leq \tau^2b_1(\Li+1)\exp{(\tau^2b_2(\Li+1)+2)}\int_{\R^d}| x|^2 \mathrm{d}\rho_0( x) < \infty.
\end{align}
We can transfer this estimate back to $(u_t,D_t)$ using Jensen's inequality and Lemma \ref{lem:estimate_weightedcov} respectively, resulting in
\begin{align}
    |u_t|^2 &= \tau^2|\mathcal{M}_\beta(\rho_t)|^2 \leq \tau^2 (b_1+b_2\int_{\R^d}| x|^2 \mathrm{d}\rho_t( x)) \sw{ \le M} < \infty \\
    \text{and}\quad \|D_t^2\|^2_F &= \tau^2 \|{\mathcal{C}_\beta(\rho_t)}\|_F^2 \leq \tau^2(b_1+b_2\int_{\R^d}| x|^{2} \mathrm{d}\rho_t( x))\sw{ \le M}  < \infty,
\end{align}
and thus get upper bounds for $u$ and $D$ by considering the supremum norms $\|u\|_\infty := \sup_{t\in[0,T]}|u_t|$ and $\|D\|_\infty := \sup_{t\in[0,T]}\|D_t\|_F$. Hence, we have verified all conditions to apply Theorem~\ref{thm:schaefersfpt} implying the existence of a fixed point of the mapping $\mathcal T$. 
\end{proof}

\begin{proof}[Proof of Theorem~\ref{thm:mckeanuniqueness}]
We showed above that a fixed point $(u,D)\in C([0,T],\R^d\times\R^d)$ of $\mathcal T$ satisfies $\sw{\|u\|_\infty,\|D\|_\infty\le M}$. Now let $(u,D), (\hat{u},\hat{D})$ be two fixed points of $\mathcal{T}$.
\sw{Recall, that} we have
\begin{align}
    \sw{\|u\|_\infty, \|D\|_\infty  \le  M}, \underset{t\in[0,T]}{\sup}\int_{\R^d}| x|^4 \mathrm{d}\rho_t( x)\leq K < \infty,\\
    \sw{\|\hat{u}\|_\infty, \|\hat{D}\|_\infty \le M}, \underset{t\in[0,T]}{\sup}\int_{\R^d}| x|^4 \mathrm{d}\hat{\rho}_t( x) \leq K < \infty 
\end{align}
Taking the difference of the paths of their corresponding processes $(Y_t)_{t\in[0,T]}$, $(\hat{Y}_t)_{t\in[0,T]}$ yields 
\begin{align}
    Y_t-\hat{Y}_t &=: z_t = z_0 - \int_0^tz_s \mathrm{d}s + \int_0^t(u_s-\hat{u}_s) \mathrm{d}s + \sqrt{2\Li} \int_0^t(D_s-\hat{D}_s) \mathrm{d}W_s. 
\end{align}
We apply It\^o's formula to obtain
\[|z_t|^2 = |z_0|^2 - \int_0^t 2|z_s|^2\,\mathrm{d}s + \int_0^t 2\langle z_s, u_s-\hat u_s\rangle \, \mathrm{d}s +2\lambda^{-1} \int_0^t \Tr{(D_s-{\hat{D}_s})({D_s}-{\hat{D}_s})^\top}\, \mathrm{d}s + \int_0^t \,\mathrm{d}M_s, \]
where $M_s$ is a Martingale such that $\E{\int_0^t \, \mathrm{d}M_s = 0}$.
We apply Cauchy-Schwarz and Young's inequality to get
\[\langle z_s, u_s-\hat u_s\rangle \le \frac12|z_s|^2 + \frac12|u_s-\hat u_s|^2\,.\]
Taking the expectation it follows that
\begin{align}
    \E{|z_t|^2} &\le \E{|z_0|^2}- \int_0^t \E{|z_s|^2}\,\mathrm{d}s + \int_0^t \E{|u_s-\hat u_s|^2}\,\mathrm{d}s + 2\lambda^{-1} \int_0^t \E{ \Tr{({D_s}-{\hat{D}_s})({D_s}-{\hat{D}_s})^\top}}\,\mathrm{d}s\,.
\end{align}
Note that we have interchanged the integral $\int_0^t$ and the expectation $\mathbb E$ using Fubini's theorem. Moreover, $\Tr{({D_s}-{\hat{D}_s})({D_s}-{\hat{D}_s})^\top} = \|{D_s}-{\hat{D}_s} \|_{F}^2$.)
Substituting $\M{\rho_t}$ for $u_t$, $\M{\hat\rho_t}$ for $\hat{u}$, $\sqrt{\C{\rho_t}}$ for $D_t$, and $\sqrt{\C{\hat\rho_t}}$ for $\hat{D}_t$, and using Lemma \ref{lem:substitutestability} and \cite[Lemma 3.2]{cboanalysis}, we obtain
\begin{align}
    \E{|z_t|^2} &\le \E{|z_0|^2}- \int_0^t \E{|z_s|^2}\,\mathrm{d}s + \int_0^t \E{|\M{\rho_s}-\M{\hat\rho_s}|^2}\,\mathrm{d}s \\
    &\quad + 2\lambda^{-1} \int_0^t \E{\|\sqrt{\C{\rho_s}}-\sqrt{\C{\hat\rho_s}}\|_F^2}\,\mathrm{d}s\\
    &\leq \E{|z_0|^2} +(c_0^{\sw{2}}+2\Li c_1\frac{\sw{4d}}{\sw{\bar \sigma}}-1)\int_0^t\E{|z_s|^2}\mathrm{d}s\,.
\end{align}
\sw{Here we used Assumption~\ref{assumpt:evs} and Lemma~\ref{lem:auxiliary_sqrt_cov_fixed} to deduce that 
\begin{align*} 
\|\sqrt{\C{\rho_s}}-\sqrt{\C{\hat\rho_s}}\|_F^2\le d \|\sqrt{\C{\rho_s}}-\sqrt{\C{\hat\rho_s}}\|_2^2 &\le 4d \|\C{\hat\rho_s}^{-1}\|_2 \|\C{\rho_s}-\C{\hat\rho_s}\|_2^2 \\ &\le \frac{4d}{\bar\sigma} \|\C{\rho_s}-\C{\hat\rho_s}\|_F^2\,.
\end{align*}}
Applying Grönwall's inequality and using the fact that $\E{|z_0|^2}=0$ yields the desired $\E{|z_t|^2}=0$, implying uniqueness.
\end{proof}

\begin{proof}[Proof of Prop~\ref{prop:linearizedfpeunique}]
  The existence follows by applying the It\^{o}'s formula to the law of the solution of \eqref{eq:linearizedsde}.

    To obtain the uniqueness, consider for each $t_0\in(0,T]$ and compactly supported smooth test function $\psi\in C_c^\infty(\R^d)$ the Kolmogorov backward equation 
    \begin{align}
        \frac{\partial h_t}{\partial t} &= -\Li \sumi{i=1}{d}\sumi{k=1}{d} (C_t)_{ik} \frac{\partial^2}{\partial x_i\partial x_k}(h_t)_k-(x-u_t)\cdot\nabla\rho_t,\quad (t,x)\in[0,t_0]\times\R^d;\quad h_{t_0} = \psi
    \end{align}
    By \cite[Thm 2.7]{pavliotis}, this linear PDE admits the solution 
    \begin{align}
        \label{eq:bwd_kolm_soln}
        h_t(x) &= \E{\psi(\theta_{t_0}^{t,x})}, t\in[0,t_0],
    \end{align}
    where $(\theta_{t_0}^{t,x})_{0\leq t \leq s \leq t_0}$ is the strong solution to the linear SDE
    \begin{align}
        \mathrm{d}\theta^{t,x}_s &= -(\theta^{t,x}_s-u_s) \mathrm{d}s + \sqrt{2\Li C_s} \mathrm{d}B_s,\qquad \theta_t^{t,x} = x
    \end{align}
    The chain rule provides, for each $(t,x)\in[0,t_0]\times\R^d$:
    \begin{align}
        \nabla_kh_t(x) &= \E{\nabla_k\psi(\theta_{t_0}^{t,x})\nabla_k\theta_{t_0}^{t,x}}.
    \end{align}
     According to \cite[Thm. 4.2]{metivier} we may apply the derivative w.r.t. parameter $k$ to the drift and diffusion coefficients of $d\theta^{t,x}_{t_0}$ separately, and simply obtain
    \begin{align}
        \mathrm{d}\nabla_k(\theta_s^{t,x})_k &= \nabla_k(\theta^{t,x}_s)_k \mathrm{d}s,\quad \nabla_k(\theta^{t,x}_t)_k = 1.
    \end{align}
     We obtain a geometric brownian motion without diffusion and have
     \begin{align}
         \nabla_k(\theta^{t,x}_s)_k =  \exp{(s-t)}
     \end{align}
    Therefore, there exists a constant $c>0$ depending only on $\psi$ s.t.
    \begin{align}
        \underset{(t,x)\in[0,t_0]\times\R^d}{\sup} |\nabla_kh_t(x)| &\leq c\exp{(T)} < \infty, k=1,\ldots,d
    \end{align}
    Furthermore, for $0\leq t < t+\delta < t_0$, we have $\theta_{s}^{t,x} = \theta_{t_0}^{t+\delta,\theta_{t_0}^{t+\delta,x}}$ for $t+\delta < s \leq t_0$ and thus 
    \begin{align}
        \frac{h_{t+\delta}-h_t(x)}{\delta} &= \frac{1}{\delta}\E{\psi(\theta_{t_0}^{t+\delta,x})-\psi(\theta_{t_0}^{t,x})} \\   
        &= \frac{1}{\delta} \E{\psi(\theta_{t_0}^{t+\delta,x})-\psi(\theta_{t_0}^{t+\delta,\theta^{t,x}_{t+\delta}})} \\
        &=\frac{1}{\delta} \E{h_{t+\delta}(x)-h_{t+\delta}(\theta^{t,x}_{t+\delta})} \\
        &=\frac{1}{\delta}\Ec{-\int_t^{t+\delta}\Li \sumi{i=1}{d}\sumi{k=1}{d}(C_s)_{ik} \frac{\partial^2}{\partial x_i\partial x_k}h_{t+\delta}(\theta^{t,x}_s)-(x-u_t)\cdot\nabla h_{t+\delta}(\theta^{t,x}_s) \mathrm{d}s}
    \end{align}
    Since we saw that the law $\rho_t=\mathbb{P}^{\theta_t}$ is a classical solution to the linearized FPE \eqref{eq:linearizedfpe}, the limit $\underset{\delta\to0}{\lim} \frac{h_{t+\delta}-h_t(x)}{\delta}$ exists.

    Now suppose $\overline\rho^1$ and $\overline\rho^2$ are two weak solutions of \eqref{eq:linearizedfpe} with identical initial condition $\rho_0^1=\rho_0^2$. We set \mk{$\delta\rho_t=\rho^1_t-\rho^2_t$}. Now we integrate the solution $h$ to the Kolmogorov backward PDE, which we have seen exists, against this (signed) measure: 
    \begin{align}
         \int h_{t_0} d\delta\rho_{t_0}(x) &= 
            \int_0^{t_0}\int_{\R^d}\partial_sh_s(x) \mathrm{d}\delta \rho_s(x)\mathrm{d}s +\int_0^{t_0}\int_{\R^d}\Li \sumi{i=1}{d}\sumi{k=1}{d} (C_s)_{ik} \frac{\partial^2}{\partial x_i\partial x_k}h_s \mathrm{d}\delta\rho_s(x)\mathrm{d}s\\
            &\qquad+ \int_0^t\int_{\R^d}(x-u_s)\cdot\nabla h_s \mathrm{d}\delta\rho_s(x)\mathrm{d}s \\
            &= \int_0^{t_0}\int_{\R^d}\partial_sh_s(x) \mathrm{d}\delta\rho_s(x)\mathrm{d}s + \int_0^{t_0}\int_{\R^d}-\partial_sh_s(x) \mathrm{d}\delta\rho_s(x)\mathrm{d}s \\
            &= 0
    \end{align}
    Since we chose $h_{t_0}=\psi\in C_c^\infty(\R^d)$, we have $\int_{\R^d}\psi(x)\mathrm{d}\delta\rho_{t_0}(x)=0$, which means $\delta\rho_{t_0}=0$. Since $t_0$ was arbitrary, we have $\overline\rho^1=\overline\rho^2$.
\end{proof}

\begin{proof}[Proof of Corollary~\ref{cor:uniqueness}]
    We construct two linear processes ($\Hat{\theta}^l_t)_{t\in[0,T]} (l=1,2)$ satisfying
    \begin{align}
    \label{eq:linearizedsde}
    d\Hat{\theta}^l_t = -(\Hat{\theta}^l_t-\M{\rho_t^l})dt + \sqrt{2\Li\C{\rho_t^l}}dW_t,
    \end{align}
    with $\Hat{\theta}_0\sim\rho_0$. The laws of these linearized processes $\Hat{\rho}^l_t:=\mathbb{P}^{\Hat{\theta}^l_t}$ are thereby weak solutions to the linear PDE
    \begin{align}
        \label{eq:linearizedfpe}
        \frac{\partial \Hat{\rho}^l_t}{\partial t} &= \Li \sumi{i=1}{d}\sumi{k=1}{d}\frac{\partial^2}{\partial x_i\partial x_k}[\C{{\rho}^l_t}_{ik}\Hat{\rho}^l_t] - \nabla\cdot((x-\M{{\rho}^l_t})\Hat{\rho}^l_t).
    \end{align}
    By assumption, $\rho^l$ also solves the above PDE, since in this case it is identical to the nonlinear FPE. In the subsequent Proposition \ref{prop:linearizedfpeunique}, we show that the weak solution to the above SDE is unique. From these two facts follows that $\Hat{\rho}^l=\rho^l$. But $\rho^l=\mathbb{P}^{\overline{\theta}^l}$ is the law of the McKean Process
    \begin{align}
        \mathrm{d}\overline{\theta}_t = -(\overline{\theta}_t-\M{\rho^l_t}) \mathrm{d}t + \sqrt{2\Li\C{\rho^l_t}} \mathrm{d}W_t,
    \end{align}
    which is solvable uniquely up to $\mathbb{P}$-indistinguishability according to Theorem \ref{thm:mckeanuniqueness}.
    Thus, $(\Hat{\theta}^l_t)_{t\in[0,T]} = (\overline{\theta}^l_t)_{t\in[0,T]}$ is a solution to the McKean Process in \eqref{eq:CBS_McKean}. Therefore
    \begin{equation}
        0 = \underset{t\in[0,T]}{\sup}\E{|\overline{\theta}^1_t-\overline{\theta}^2_t|} 
          = \underset{t\in[0,T]}{\sup}\E{|\Hat{\theta}^1_t-\Hat{\theta}^2_t|} 
          \geq \underset{t\in[0,T]}{\sup} W_2(\Hat{\rho}^1_t,\Hat{\rho}^2_t) 
          \geq \underset{t\in[0,T]}{\sup} W_2({\rho}^1_t,{\rho}^2_t),
    \end{equation}
    which is the desired result.
\end{proof}

\section{{Additional results used in the proof}}\label{app:theorems}

\begin{theorem}[{\cite[Prop 2.2(ii)]{sznitman}}] \label{thm:chaos} Let $E$ be a Polish Space, and $\rho^J=\frac1J\sumi{i=1}{J}\delta_{\theta^{j,J}}$ be $\PM(E)$-valued random variables, where $(\theta^{1,J},\dots,\theta^{J,J})$ is distributed according to $u_J$. Assume that $(u_J)_{J\in\mathbb N}$ is a sequence of symmetric probabilities on $E^N$. Then, $(\rho^J)_{J\in\mathbb N}$ is tight if and only if $(\mathbb P^{\theta^{1,J}})_{J\in\mathbb N}$ is tight.
\end{theorem}

\begin{theorem}[{\cite[Thm 1]{aldous}}] \label{thm:aldous}  Let $\{X^n\}_{n\in\mathbb{N}}$ be a sequence of stochastic processes on a common probability space $(\Omega,\mathcal{F},\mathbb{P})$ with $X^n: \Omega\to C([0,T];\mathbb{R}^d)$. The sequence is tight on $C([0,T];\mathbb{R}^d)$ if the following two conditions hold:

 \begin{enumerate}
     \item  $\{\mathbb{P}^{X^n_t}\}_{n\in\mathbb{N}}$ is tight on $\R^d$ for every $t\in[0,T]$.
     \item $\forall\epsilon> 0,\eta> 0$ and $m$ there exists $\delta_0\in(0,T)$ and $n_0\in\mathbb{N}$ s.t., when $ n\geq n_0$, $\delta\leq\delta_0$ and $\tau$ is a $\sigma(X^n_s,s\in[0,T])$-stopping time with finite range and $\tau\leq m$, then:
     \begin{equation}\label{eq:aldous2}
         \mathbb{P}(|X^n_{\tau+\delta}-X_\tau^n|\geq\epsilon) \leq \eta.
     \end{equation}
 \end{enumerate}
\end{theorem}

The solvability of the McKean-Vlasov type SDE is obtained by use of the following Theorem.
 
\begin{theorem}[{\cite[Thm 11.3]{gilbarg}}] \label{thm:schaefersfpt} Schaefer's Fixed Point Theorem: Let $\mathcal{T}:\mathfrak{B}\to\mathfrak{B}$ be a compact mapping of a Banach Space $\mathfrak{B}$ into itself and suppose there is a constant $M$ such that if $x\in\mathfrak{B}$ and $\tau\in[0,1]$ with $x=\tau\mathcal{T}x$, then $\|x\|_\mathfrak{B}<M$. 
Then $\mathcal{T}$ has a fixed point. 
\end{theorem}

\end{document}